\documentclass[12pt,oneside,reqno]{amsart}
\usepackage{amsmath,amssymb, amsthm,mathtools}
\usepackage[parfill]{parskip}
\usepackage[margin=1in]{geometry}
\usepackage{enumitem}
\usepackage{tikz}
\usepackage{tikz-cd} 
\usepackage{dsfont}
\usepackage{mathrsfs}
\usepackage{bbm}
\usepackage{dirtytalk}
\usepackage{bm}
\usepackage{microtype}
\usepackage{caption}
\usepackage{subcaption}
\usepackage{array}
\usepackage{longtable}
\usepackage{nicefrac}

\newcommand{\rank}{\text{rank }}

\newtheorem{theorem}{Theorem}
\newtheorem{lemma}[theorem]{Lemma}
\newtheorem{definition}[theorem]{Definition}
\newtheorem{prop}[theorem]{Proposition}

\newtheorem{corollary}[theorem]{Corollary}

\newcommand{\N}{\mathbb{N}}

\newcommand{\R}{\mathbb{R}}

\newcommand{\T}{\mathbb{T}}
\newcommand{\Z}{\mathbb{Z}}

\newcommand{\E}{\mathbb{E}}

\DeclareMathOperator{\Real}{Re}

\newcommand{\paren}[1]{\left( #1 \right)}

\newcommand{\brac}[1]{\left[ #1 \right]}

\newcommand{\abs}[1]{\left\vert#1\right\vert}

\newcommand{\set}[1]{\left\{#1\right\}}

\newcommand{\nlty}{\text{nullity }}

\DeclareMathOperator{\im}{im}
\DeclareMathOperator{\support}{support}

\newcommand{\mvleftarrow}[1]{\hspace{1.8em} \mathclap{\xleftarrow{\hspace{1em} \mathclap{#1}\hspace{1em}}}\hspace{1.8em}}

\usepackage{scalerel,stackengine}
\newcommand\pig[1]{\scalerel*[5.5pt]{\Big#1}{%
  \ensurestackMath{\addstackgap[1.5pt]{\big#1}}}}
\newcommand\pigl[1]{\mathopen{\pig{#1}}}
\newcommand\pigr[1]{\mathclose{\pig{#1}}}

\usepackage{pgfplots}
\pgfplotsset{compat=1.6}
\usetikzlibrary{calc,quotes,angles,arrows.meta,positioning, fillbetween,decorations,intersections,pgfplots.fillbetween,calligraphy,decorations.markings}
\usepackage{hyperref}

\title[A Cellular Representation of the Potts Lattice Higgs Model]{A Cellular Representation of \\ the Potts Lattice Higgs Model}

\author{Summer Eldridge}
\email{seldridge@gradcenter.cuny.edu}
\address{Department of Mathematics, Graduate Center, City University of New York, 365 5th Ave, New York, NY 10016, USA}
\author{Malin P. Forsstr{\"o}m}
\email{palo@chalmers.se}
\address{Mathematical Sciences, Chalmers University of Technology and University of Gothenburg, SE-412 96 G\"oteborg, Sweden}
\author{Benjamin Schweinhart}
\email{bschwei@gmu.edu}
\address{Department of Mathematical Sciences, George Mason University, Fairfax, VA 22030, USA}
\begin{document}

\begin{abstract}
The $i$-dimensional Potts lattice Higgs model is a random assignment of spins in $\Z_q$ to the $i$-dimensional cells of a cell complex induced by a Hamiltonian with a Potts interaction on the $(i+1)$-cells and an additional term playing the role of an external field. We develop a representation of this model as a pair of dependent plaquette percolations, and prove that Wilson line expectations can be expressed in terms of the probability of a topological event. As an application, we prove the existence of a phase transition for the Marcu--Fredenhagen ratio in the Potts lattice Higgs model on $\Z^d$ when $i=1.$    
\end{abstract}
\maketitle

\section{Introduction}

Lattice Higgs models, also called lattice gauge theories with matter, are a class of probability distributions studied in physics as discretized models of gauge fields in the presence of particles~\cite{FS79}. Models with $SU(n)$-valued fields are of the greatest interest, but simpler analogues are obtained by taking spins in the integers modulo $q$; the complexity is further reduced with a Potts interaction, as defined in~\cite{B88}. In the special case with $\mathbb{Z}_2$-valued fields, these models were already introduced in~\cite{wegner}, as examples of spin models with phase transitions and local symmetries, and also appear in quantum information theory as natural noised version of Kitaevs toric code, see, e.g., \cite{NSS21}
We will throughout refer to these models as Potts lattice Higgs models.

In this paper, we develop an analogue of the Fortuin-Kasteleyn random cluster representation (see, e.g.,~\cite{FK72} for Potts lattice Higgs models, both generalizing the random cluster representation of the Potts model with external field~\cite{G16} and building on recent work establishing a cellular representation of Potts lattice gauge theory~\cite{duncan2025sharp,duncan2025topological,shklarov,bernoulli-cell-complexes}. 
More specifically, we define a coupling between the $i$-dimensional Potts lattice Higgs model --- which assigns spins in $\Z_q$ to the $i$-dimensional cells of a cell complex --- and a new construction called \emph{coupled plaquette percolation} (CPP), consisting of a dependent pair of plaquette percolations of dimensions $(i+1)$ and $i.$ We prove that Wilson line expectations in the former can be expressed in terms of the probability of a topological event in the latter. In addition, the CPP on $\Z^d$ has a duality transformation defined on the level of states; this gives a concrete, geometric interpretation of the duality of the partition functions of the Potts lattice Higgs model on $\Z^3.$ As an application of the CPP, we prove the existence of a phase transition for the Marcu--Fredenhagen ratio. To simplify the arguments and keep the required knowledge of algebraic topology to a minimum, we assume that $q$ is prime and take free or periodic boundary conditions; extensions to more general models can be obtained via modifications similar to those in~\cite{prcm-24}.   
%
%
Of particular interest is the case where $i=1$ and $X$ is a finite subcomplex of the cubical complex $\Z^d$ formed by tesselating Euclidean space with unit cubes. In this context, the Potts lattice Higgs model assigns spins to the edges with a Potts field on the edges and a Potts interaction on the signed sum of edges around every plaquette. The CPP is defined on a pair of percolations: a $2$-plaquette percolation $P_2$ and a bond percolation $P_1,$ where the probability is weighted by the size of the relative one-dimensional cohomology with coefficients in $\Z_q$, denoted $H^1(P_2,P_1;\Z_q)$. This counts the number of spin assignments to the edges of $X$ that are compatible with the pair $\paren{P_2,P_1}$ in the sense that edges in $P_1$ are assigned spin $0$ the oriented sum of spins around each plaquette of $P_2$ adds up to $0$ modulo $q.$ When $d=3,$ the model is self-dual: there's a bijective measure-preserving mapping between the CPP and itself with different parameters.

Special cases of the Ising lattice Higgs model include the Ising model with external field ($i=0$ and $q=2$) and the Ising lattice Higgs model ($i=1$ and $q=2).$ The latter is an object of previous study in both the physics~\cite{FS79,HS91, NSS21,stahl2026slow} and in the mathematics~\cite{p2024phase,F24,F25b} literature. In particular, in~\cite{p2024phase}, it is shown that this model has a phase transition, which will be extended to \( q \neq 2 \) in this paper. Alternate surface representations have been developed for these models: a membrane expansion obtained via a high temperature expansion~\cite{HS91} and a random current expansion~\cite{F25a,A25}. In addition, we note that the coupling between the Ising lattice Higgs model and the CPP may be of interest to physicists studying the former model in computational experiments; it suggests a Swendsen--Wang-type algorithm~\cite{sw,edwards-sokal,pizzimenti2025generalized} whose dynamics may change along the phase boundaries and which may converge faster than other Monte Carlo algorithms in practice.


\subsection{Definitions and Notation}\label{sec:defs}

We start by introducing terminology and notation from algebraic topology. The cell complex $\Z^d$ is the tesselation of Euclidean space by unit cubes with integter points as vertices. $\Z^d$ is composed of $j$-dimensional cells (called \emph{$j$-cells} or \emph{$j$-plaquettes})  for $0\leq j\leq d$ where the $j$-cells are the $j$-dimensional faces of cubes in the tesselation. That is, $0$-cells are vertices, $1$-cells are edges, $2$-cells are unit squares, and so on. A $j$-dimensional subcomplex of $\Z^d$ is a union of cells of $\Z^d$ so that the dimension of each cell is at most $j.$ Note that if $\sigma$ is cell of a subcomplex $X$ and $\tau$ is a face of $\sigma,$ then $\tau$ is also a cell of $X.$  Another cell complex is the discrete torus $\T_N^d$ obtained from the subcomplex $\brac{0,N}^d$ of $\Z^d$ by identifying opposite faces. For a cell complex $X$ denote by $X^{\brac{j}}$ the collection of all $j$-cells of $X$ and the \emph{$i$-skeleton} $X^{\paren{j}},$ the subcomplex of $X$ consisting of all cells of dimensions less than or equal to $j.$ 

Given a cell complex $X$, we notate by $C^j(X; \Z_q)$ the group of functions $f$ from oriented $j$-cells in $X$ to the additive group $\Z_q$ of integers modulo $q$, with the property that if $-x$ is $x$ considered with the opposite orientation, $f(-x)=-f(x)$.  An element of $C^j(X;\,\Z_q)$ is called a ``\emph{$j$-cochain}" or a ``discrete $j$-form''. Dual to cochains are \emph{chains} which form the group $C_j(X;\,\Z_q)$ of formal sums of oriented $j$-cells with coefficients in $\Z_q$, with the relation that $-x$ is obtained from $x$ by reversing the orientation. By extending linearly, we can evaluate cochains on chains. 

The chain groups come with linear \emph{boundary maps} $\partial_j\colon C^j(X;\,\Z_q)\to C^{j-1}(X;\,\Z_q)$ defined by $\partial_j(x)=\sum y_k$, where $y_k$ are the set of oriented $(j-1)$-cells incident to $x$. The linear \emph{coboundary maps} $\delta_j\colon C^j(X;\,\Z_q)\to C^{j+1}(X;\,\Z_q)$ are defined by $\delta_j(f)(x)=f(\partial_{j+1}(x))$. Subscripts on the maps are generally omitted. The boundary and coboundary maps are used to defined homology and cohomology, as described at the beginning of Section~\ref{sec:topos} below. 

\begin{definition}[Potts lattice Higgs Model]%
Fix $i\in \Z_{\geq 1},q\in\Z_{\geq 2},{\beta_2}>0$ and ${\beta_1}>0,$ and let $X$ be a finite cell complex. The \emph{Potts lattice Higgs model} is the Gibbs measure $\mu=\mu_{{\beta_2},{\beta_1},q,i,X}$ on $C^{i}(X;Z_q)$ induced by the Hamiltonian
$$\mathcal{H}\paren{f}=-{\beta_2} \sum_{\sigma \in X^{\brac{i+1}}}I_0\paren{\delta f\paren{\sigma}}-{\beta_1} \sum_{\epsilon\in  X^{\brac{i}}}I_0\paren{f\paren{\epsilon}},$$
where $I_x(y)$ is the indicator that $y=x.$
\end{definition}

 If $i=1$ and $q =2,3$ this model is equivalent to the $\Z_q$ lattice Higgs model with unitary gauge~\cite{FS79,F24} with Hamiltonian
$$-\beta_2'\sum_{\sigma\in X^{\brac{2}}}\Real e^{\frac{2\pi i}{q} df(\sigma)}-\beta_1'\sum_{\epsilon\in X^{\brac{1}}}\Real e^{\frac{2\pi i}{q}f(\epsilon)}.$$
This follows from the fact that for $q=2,3$, and $x\in \Z_q$, $x \mapsto \Real{e^{(\frac{2\pi i}{q}x)}}$ is an affine transformation of $I_0(x)$.

We note that there is an alternative definition of the Potts lattice Higgs Model~\cite{B88} which assigns probabilities to a pair $\paren{f,g}$ with $f\in C^{i}(X;\,\Z_q)$ and $g\in C^{i-1}(X;\,\Z_q).$ Specifically, one can consider the Gibbs measure $\nu$ induced by the Hamiltonian 
\begin{equation}
\label{eq:generalgauge}
    \tilde{\mathcal{H}}\paren{f,g}=-{\beta_2'} \sum_{\sigma \in X^{\brac{i+1}}}I_0\paren{\delta f\paren{\sigma}}-{\beta_1'} \sum_{\epsilon\in  X^{\brac{i}}}I_0\paren{f\paren{\epsilon}-\delta g\paren{\epsilon}}.
\end{equation}
We call $\nu$ the Potts lattice Higgs model \emph{with general gauge}. 
 For any $h\in C^{i-1}(X;\,\Z_q)$, $\nu(f,g)$ is invariant under the gauge transformation $f\to f+\delta h,\; g\to g+h$. The map $(f,g)\to (f-\delta g,0)$ covers its image evenly, so any function invariant under such a transformation (e.g. a Wilson line variable) will have the same expectation in both models. While Wilson line variables are not gauge invariant, the Wilson line expectation for $\mu$ equals the expectation of its analogue $$W_\gamma'(f,g) \coloneqq \exp\paren{\frac{2\pi i}{q}(f(\gamma)-g(\partial\gamma)}$$ for $\nu$ (see, e.g., ~\cite{F24}).

Next, we define the CPP as a dependent pair of random percolation subcomplexes. 
\begin{definition}[Percolation Subcomplex]
 A $j$-dimensional \emph{percolation subcomplex} of $X$ is a subcomplex $P$ of $X$ satisfying
 $$X^{\paren{j-1}}\subseteq P \subseteq X^{\paren{j}}.$$
That is, $P$ contains all cells of dimension less than $j$ and a subset of the the $j$-dimensional cells. When convenient, we treat a $j$-dimensional percolation subcomplex $P$ as a binary assignment to the $j$-cells, with $P(x)=1$ iff $x\in P$ for $x\in X^{\brac{j}}.$ We will write $\abs{P}$ for the number of $j$-cells of $P.$
\end{definition}
\begin{definition}\label{defn:compatible}
Let $(P_2,P_1)$ be a pair of percolation subcomplexes, where the dimensions of $P_2$ and $P_2$ are $(i+1)$ and $i$ respectively. An $i$-cochain ${f\in C_i(X;\,\Z_q)}$ is \emph{compatible} with $\paren{P_2,P_1}$ if $f\paren{\epsilon}=0$ for all $i$-cells $\epsilon$ of $P_1$ and $\delta f\paren{\sigma}=0$ for all $(i+1)$-cells $\sigma$ of $P_2.$  
\end{definition}

The cochains compatible with $\paren{P_2,P_1}$ form an additive group, the relative cocycle group, denoted $Z^i\paren{P_2,P_1;\,\Z_q}.$ As we will explain below, this group coincides with the $i$-dimensional relative cohomology $H^{i}(P_2,P_1;\mathbb{Z}_q)$  because $P_1$ is a percolation subcomplex. For now, we only need that the symbol $\abs{H^{i}(P_2,P_1;\mathbb{Z}_q)}$ counts the number of compatible $i$-cochains. Our cellular representation weights a pair of percolation subcomplexes by this quantity.

\begin{definition}[Coupled Plaquette Percolation (CPP)]
\label{def:CPP}
For a finite cell complex $X,$ a prime integer $q,$ a non-negative integer $i$, and $0\leq p_2,p_1\leq 1$ define the \emph{Coupled Plaquette Percolation} measure $\rho=\rho_{p_2,p_1,q,i,X}$ on
 percolation subcomplexes \( P_2\) and \( P_1\) of dimensions \( (i+1) \) and \( i\) respectively, by
$$\rho\paren{P_2,P_1}\propto   p_2^{|P_2|}\paren{1-p_2}^{|X^{\brac{i+1}}|-|P_2|} p_1^{|P_1|}\paren{1-p_1}^{|X^{\brac{i}}|-|P_1|} \abs{H^{i}(P_2,P_1;\mathbb{Z}_q)}.$$

\end{definition}

We will often use the convenient reparametrization $k_2=\frac{p_2}{1-p_2}$ and $k_1=\frac{p_1}{1-p_1}.$ In these coordinates, the law of the CPP  $\rho=\rho\paren{k_2,k_1,q,i,X}$
is 
$$
    \rho\paren{P_2,P_1}\propto  k_2^{|P_2|}  k_1^{|P_1|}\abs{H^{i}(P_2,P_1;\mathbb{Z}_q)}.
$$
The CPP has a number of interesting special cases. When $p_1=0$, then $P_2$ is distributed as the $(i+1)$-dimensional plaquette random cluster model with coefficients in $\Z_q$~\cite{bernoulli-cell-complexes}. Moreover, when $X$ is simply connected and $p_2=1$, then $P_1$ is distributed as the $i$-dimensional PRCM with parameter $p_1.$ On the other hand, when either $p_1=1$ or $p_2=0$ the appropriate marginal distributions are independent plaquette percolations (see Proposition~\ref{prop:10} below). Fixing a states for either complex results in a random complex that can be interpreted as a weighted plaquette random cluster model.

\subsection{Main Results} \label{sec:main}

We begin by establishing the coupling between the Potts lattice Higgs model and the CPP. Denote by $\Xi^j=\Xi^j\paren{X}$ the collection of $j$-dimensional percolation subcomplexes of~$X.$

\begin{theorem} \label{thm:coupling} 
Let $X$ be a finite cell complex, $q$ be a prime integer, and $\beta_2,\beta_1\geq 0.$ If \(p_2=1-e^{-{\beta_2}},\) \( p_1=1-e^{-{\beta_1}}\),  \( k_2=\frac{p_2}{1-p_2}=e^{\beta_2}-1,\) and \( k_1=\frac{p_1}{1-p_1}=e^{\beta_1}-1\), then there is a coupling $\kappa=\kappa_{k_2,k_1,q,i,X}$, ${\kappa \colon C^i(X;\,\Z_q)\times \Xi^{i+1}\times \Xi^{i}\to \brac{0,1}}$, defined by
\begin{align*}
    \kappa(f,P_2,P_1)&\propto \prod_{\epsilon\in X^{\brac{i}}}\pigl[I_0\bigl(P_1(\epsilon)\bigr)+k_1\cdot P_1(\epsilon)\cdot I_0\bigl(f(\epsilon)\bigr)\pigr]\\ &\qquad \cdot \prod_{\sigma\in X^{\brac{i+1}}}\pigl[I_0\bigl(P_2(\sigma)\bigr)+k_2\cdot P_2(\sigma)\cdot I_0\bigl(\delta f(\sigma)\bigr)\pigr]
\end{align*}
which satisfies the following. 
\begin{itemize}
    \item The the marginal distribution on $f$ is $\mu_{\beta_2,\beta_1,q,i,X}$ (the Potts lattice Higgs model). 
    \item The marginal distribution on $(P_2,P_1)$ is $\rho_{p_2,p_1,q,i,X}$ (the CPP). 
    \item The conditional distribution on $\paren{P_2,P_1}$ given $f$  is independent percolation with probability $p_2$ on the set of $(i+1)$-cells so that $\delta f \paren{\sigma}=0$ and probability $p_1$ on the set of $i$-cells such that~$f\paren{\epsilon}=0.$   
    \item The conditional distribution of $f$ given $(P_2,P_1)$ is the uniform measure on $Z^{i}(P_2,P_1;\mathbb{Z}_q).$
\end{itemize}
\end{theorem}

\begin{figure}[htb]
    \centering
    \includegraphics[width=.4\textwidth]{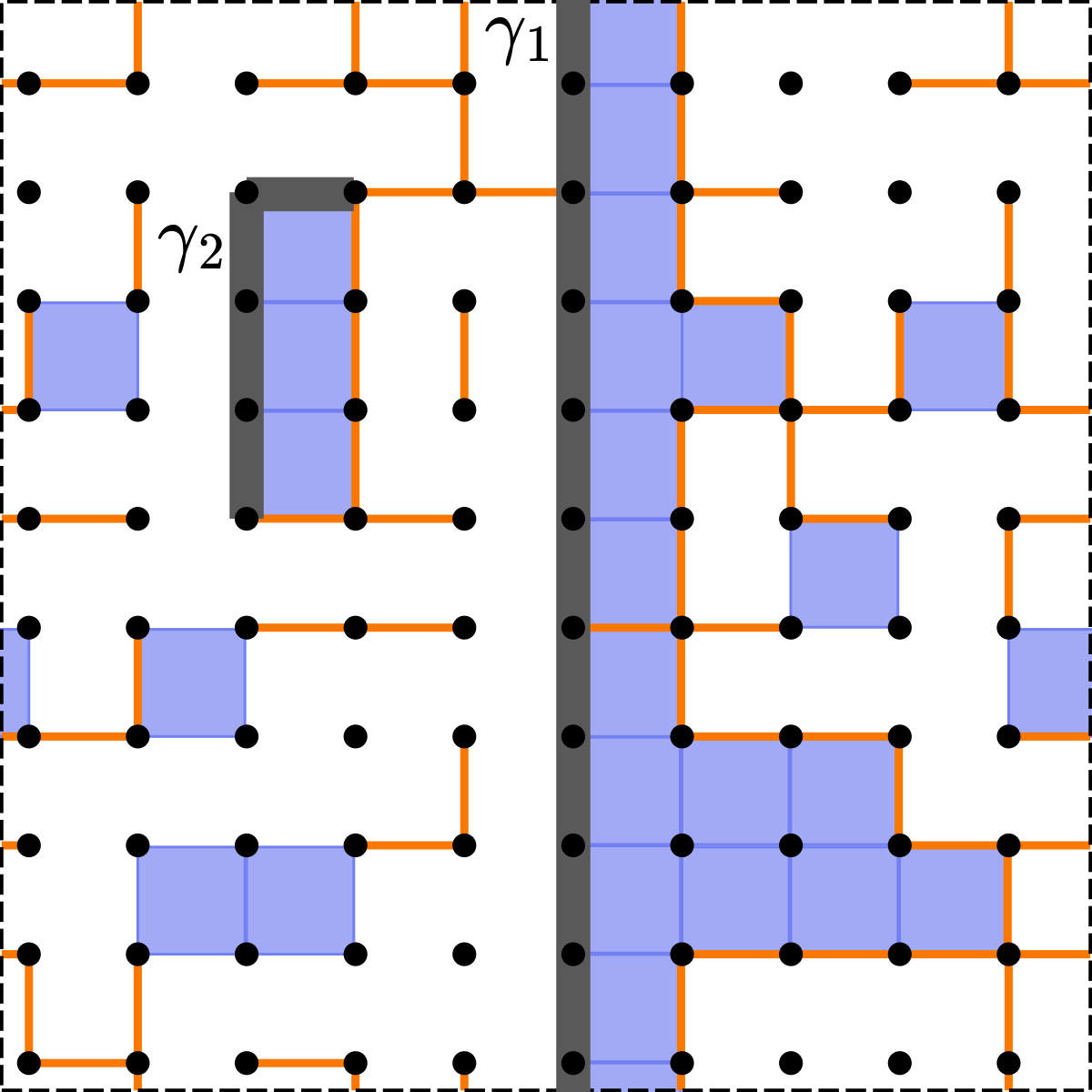}\qquad
    \includegraphics[width=.4\textwidth]{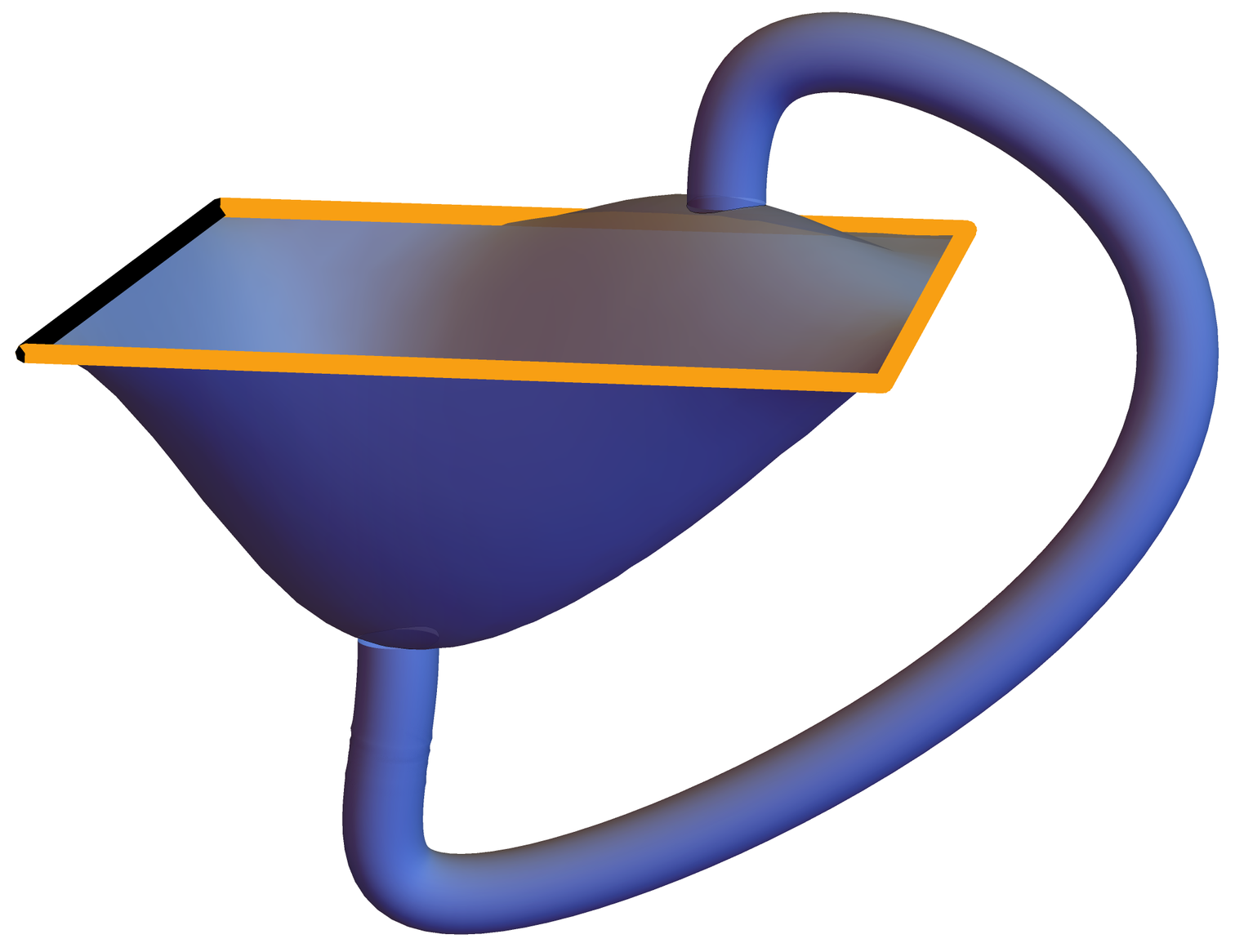}
    \caption{Two illustrations of the event $V_{\gamma}.$ On the left, $P_2$ is depicted by the collection of light blue squares and $P_1$ by the orange bonds. The events $V_{\gamma_1}$ and $V_{\gamma_2}$ both occur, where $\gamma_1$ and $\gamma_2$ are the loop and path shown in dark gray. On the right, $P_2$ is shown by the blue surface and $P_1$ by the orange line segment. $\gamma$ is depicted in black. Here, $V_{\gamma}$ occurs when homology is taken with coefficients in $\Z_2$ but the plaquettes in the orange surface cannot be compatibly oriented to yield a null-homology when coefficients are taken in a field of odd characteristic. The figure on the right was adapted from Figure 1 of~\cite{duncan2025sharp} which was in turn inspired by Figure 1 of~\cite{aizenman1983sharp}.} 
    \label{fig:vgammarelative}
\end{figure}

The above coupling can be modified to provide a cellular representation of Potts lattice Higgs model in general gauge, as explained in Section~\ref{sec:generalgauge} below. In the special case $i=0,$ the measure $\mu_{\beta_2,\beta_1,q,i,X}$ is the Potts model with external field strength~$k_1.$ Traditionally, the random cluster model is extended to a graphical representation of the Potts model with external field via the introduction of a ``ghost vertex''~\cite{G16}. We describe why this is equivalent to the $i=1$ case of the CPP in Section~\ref{sec:ghost}.


An importable observable in the Potts lattice Higgs model is the Wilson line observable, which we now define, together with a topological event which provides an analogue of the Wilson line observable for the CPP.

\begin{definition}[$V_\gamma$ and $W_\gamma$]
Fix an $i$-chain $\gamma=\sum c_j x_j\in C_i\paren{X;\Z_q}$. 
For $f \sim \mu_{\beta_2,\beta_1,q,i,X}$, we let~$W_\gamma(f)$ be the random variable 
$\exp \bigl( \frac{2\pi i}{q} f(\gamma)\bigr).$ 
%
If $i=1$, $c_j =1  $ for all \( j \), and the  edges $x_j$ form a oriented loop, then \( W_\gamma\) is referred to as a  \emph{Wilson loop variable}. Similarly, if  $i=1$, $c_j=1 $ for all \( j \), and the edges $x_j$ form an oriented path, then \( W_\gamma\) is referred to as a \emph{Wilson line variable}. 

For the measure $\rho$, $V_{\gamma}$ is the event that there exists an $(i+1)$-chain $\tau$ so that $\gamma-\partial\tau$ is supported on $P_1,$ or equivalently that $\brac{\gamma}=0$ in $H_{i}(P_2,P_1;\mathbb{Z}_q)$ (see Figure~\ref{fig:vgammarelative}), where the relative homology $H_{i}(P_2,P_1;\mathbb{Z}_q)$ is defined in Section~\ref{sec:topos} below.%
\end{definition}
 

Note that $V_\gamma$ is an increasing event; adding more $i$-cells to $P_1$ increases the number of $i$-chains in $C_i(P_1;\,\Z_q)$, while adding more $(i+1)$-cells to $P_2$ increases the number of $i$-chains that are boundaries of $(i+1)$-chains in $C_{i+1}(P_2;\,\Z_q)$.

Our next result expresses the expectation of a Wilson line observable \( W_\gamma\) in terms of the probability of the event \( V_\gamma \) in the CPP. Observe that that taking the limit as $\beta_1\to\infty$ recovers Theorem~5 of~\cite{duncan2025topological}.

\begin{theorem}\label{thm:W=V}
    Let $X$ be a finite cell complex, $q$ be a prime integer, $\beta_2,\beta_1\geq 0,$ and $\gamma\in C_i(X;\,\Z_q).$ If~$\mu=\mu_{\beta_2,\beta_1,q,i,X}$ and $\rho=\rho_{p_2,p_1,q,i,X}$, where $p_2=1-e^{-{\beta_2}}$ and $p_1=1-e^{-{\beta_1}}$, then $\mathbb{E}_{\mu}\paren{W_\gamma}=\rho\paren{V_{\gamma}}.$
\end{theorem}

Let $\mu_{\beta_2,\beta_1,q,i,\Z^d}$ be the weak limit of the measures $\mu_{\beta_2,\beta_1,q,i,\Lambda_N}$ where $\Lambda_N=\brac{-N,N}^d$ (by definition, these measures have free boundary conditions). This infinite volume limit exists by standard arguments. Our arguments can easily be extended to infinite volume limits obtained from other boundary conditions, but we restrict attention to the free measure for simplicity.  

We show that for \( i= 1\) the  Potts lattice Higgs model exhibits a phase transition marked by a change in the behavior of a ratio between Wilson line expectations. To be able to state this theorem, we first define an observable in terms of the decomposition of a loop into two paths.

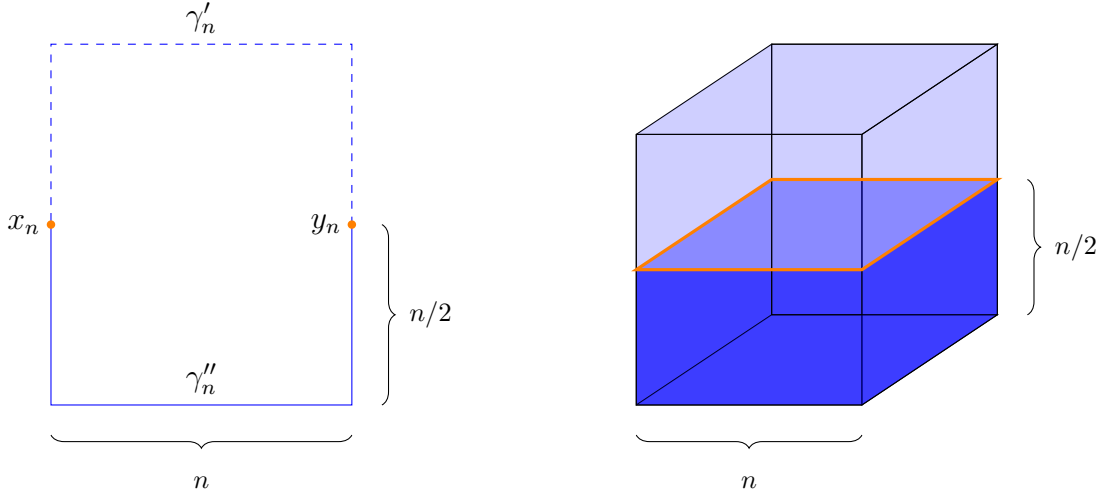
\begin{figure}[t]\centering
    \begin{subfigure}[t]{0.45\textwidth}\centering
	\begin{tikzpicture}[scale=.8] 
		
		\draw[blue] (0,0) node[left] {\color{black}$x_n$}  -- (0,-3)-- (5,-3) node[midway,above] {\color{black}$\gamma''_n$}  -- (5,0)  node[left] {\color{black}$y_n$};  
		
		\draw[dashed,blue] (5,0)  -- (5,3) -- (0,3) node[midway,above] {\color{black}$\gamma_n'$} -- (0,0);  
		\fill[orange] (0,0) circle (2pt);
		\fill[orange] (5,0) circle (2pt);
        
        \draw [decorate,decoration={brace,amplitude=5pt,mirror,raise=2.2ex}]
  (0,-3) -- (5,-3) node[midway,yshift=-2.5em]{\footnotesize $n$};
		
		\draw [decorate,decoration={brace,amplitude=5pt,mirror,raise=2.2ex}]
  (5,-3) -- (5,0) node[right, midway,xshift=1.5em]{\footnotesize $n/2$};

	\end{tikzpicture}
	
	\caption{The paths \( \gamma_n' \)~(solid) and~\( \gamma''_n \)~(dashed)  appearing in the definition of the Marcu--Fredenhagen ratio.}
	\label{figure: Wilson lines MF}
    \end{subfigure}
    \hspace{2ex}
    \begin{subfigure}[t]{0.45\textwidth}\centering
    \begin{tikzpicture}[scale=0.6]

        \draw[fill=blue,fill opacity=.1] (0,0) -- (0,3) -- (5,3) -- (5,0); 
        \draw[fill=blue,fill opacity=.1] (3,2) -- (3,5)-- (8,5) -- (8,2);
        \draw[fill=blue,fill opacity=.1] (0,0) -- (0,3) -- (3,5)--(3,2);
        \draw[fill=blue,fill opacity=.1]  (5,0) -- (5,3) -- (8,5)--(8,2);
        \draw[fill=blue,fill opacity=.1] (0,3) -- (5,3) -- (8,5) -- (3,5);

        \draw[fill=blue,fill opacity=.6] (0,0) -- (0,-3) -- (5,-3) -- (5,0);  
        \draw[fill=blue,fill opacity=.4] (3,2) -- (3,-1)-- (8,-1) -- (8,2);
        \draw[fill=blue,fill opacity=.4] (0,0) -- (0,-3) -- (3,-1)--(3,2);
        \draw[fill=blue,fill opacity=.6] (5,0) -- (5,-3) -- (8,-1)--(8,2);
        \draw[fill=blue,fill opacity=.4] (0,-3) -- (5,-3)--(8,-1) -- (3,-1);

        \draw[very thick, orange] (0,0) -- (3,2) -- (8,2) -- (5,0) -- cycle;
        
        \draw [decorate,decoration={brace,amplitude=5pt,mirror,raise=2.2ex}]
  (0,-3) -- (5,-3) node[midway,yshift=-2.5em]{\footnotesize $n$};
		
		\draw [decorate,decoration={brace,amplitude=5pt,mirror,raise=2.2ex}]
  (8,-1) -- (8,2) node[right, midway,xshift=1.5em]{\footnotesize $n/2$};

  \path[] (-1.5,0) circle (2pt);
    \end{tikzpicture}
    \caption{The oriented surfaces \( q_n'\) and \( q_n''\) that would be natural to use in a definition of the MF ratio for \( i = 2.\)}
    \end{subfigure}
    \caption{In the figures above, we draw the paths and surfaces used in the definitions of the Marcu--Fredenhagen ratio for \( i = 1\) and \( i = 2\) respectively.}
	\end{figure}

\begin{definition}[The Marcu--Fredenhagen ratio for \( i = 1\)]\label{def: MF}
   Let $q$ be a prime integer, \( i = 1, \) and $\beta_2,\beta_1\geq 0,$ and let $\mu = \mu_{\beta_2,\beta_1,q,i,\Z^d}.$
    %
    Further, let $n\in 2\N$ and set $\gamma_n=\partial \bigl(\brac{0,n}^2\times\set{0}^{d-2}$ and let $\gamma_n'=\gamma_n'\paren{R,T}$ and $\gamma_n''=\gamma_n''$ be the paths formed by the upper and lower halves of $\gamma_n$ (see Figure~\ref{figure: Wilson lines MF}). The \emph{Marcu--Fredenhagen ratio} is defined by 
    
    $$
        R\paren{\beta_2,\beta_1,n}
        =
        \frac{\mathbb{E}_{\mu}({W_{\gamma_n'}})\mathbb{E}_{\mu}({W_{\gamma_n''}})}{\mathbb{E}_{\mu}({W_{\gamma_n}})}=\frac{\mathbb{E}_{\mu}({W_{\gamma_n'}})^2}{\mathbb{E}_{\mu}({W_{\gamma_n}})}.
    $$
\end{definition} 

The quantity $R\paren{\beta_2,\beta_1,n}$ is thought to be relevant to the physics of the Euclidean lattice Higgs model~\cite{fredenhagen1988dual,bricmont1985statistical}, and is predicted to exhibit a phase transition marked by whether or not its limit is zero (see Figure~\ref{figure: phase diagram}). The asymptotics of similar ratios where $\gamma_n'+\gamma_n''$ is taken to be a growing rectangular boundary with a fixed aspect ratio are also of interest, even in the case when \( \gamma_n' \) and \( \gamma_n''\) do not have equal length, and are conjectured to have the same phase diagram \cite{GLIOZZI2006120,fredenhagen1988dual,bricmont1985statistical}. Our arguments can be modified to work in this context. 
For simplicity, we assume that $\gamma_n$ is a square. 

We note that by Theorem~\ref{thm:W=V}, we can equivalently define the Marcu--Fredenhagen ratio using the equivalent topological quantity~\( V_\gamma. \)

\begin{definition}[The topological Marcu--Fredenhagen ratio]\label{def: MF top}
 Let $q$ be a prime integer, \( i = 1, \) and $p_2,p_1\in (0,1],$ and let $\rho = \rho_{p_2,p_1,q,i,\Z^d}.$
    Further, let \( (\gamma_n',\gamma_n'')\) be as in Definition~\ref{def: MF}.
    The \emph{topological Marcu--Fredenhagen ratio} is defined by
    $$
        \hat{R}\paren{p_2,p_1,n}
        =
        \frac{\rho({V_{\gamma_n'}})\rho({V_{\gamma_n''}})}{\rho({\gamma_n})}=\frac{\rho({V_{\gamma_n'}})^2}{\rho({V_{\gamma_n}})}.
        $$
\end{definition}

	\begin{figure}[t]
    \begin{tikzpicture}[scale=1.2]
    \begin{axis}[
        axis x line=middle,axis y line=middle, ymin=-.05, ymax=1.1, xmin=-.05, xmax=1.1, ticks=none, 
        xlabel={$\beta_2$}, ylabel={$\beta_1$},
        x label style={at={(axis description cs:.85,0.05)},scale=.8},
        y label style={at={(axis description cs:0.05,.85)},scale=.8}, 
        every axis/.append style={font=\footnotesize}, ]

\addplot[domain=0.49:1, samples=201, name path=E, draw=none]{0};
\addplot[domain=0.49:1, samples=201, name path=F, draw=none]{.3*(x-.49)^(1/2)};
\addplot [fill=blue, fill opacity = .2] fill between [of=E and F];

\addplot[domain=0.09:0.091, samples=201, name path=C, draw=none]{1000*(x-0.09)};
\addplot[domain=0.0:0.001, samples=201, name path=D, draw=none]{1000*x};
\addplot [fill=blue, fill opacity = .2] fill between [of=C and D];

\addplot[domain=0:1, samples=201, name path=A,draw=none]{1};  
\addplot[domain=0:1, samples=201, name path=B,draw=none]{0.8};
\addplot [fill=blue, fill opacity = .2] fill between [of=A and B];

 \addplot[domain=0:0.35, samples=100, color=black, dashed] {0.71-2*x^2};
               
\addplot[domain=0.35:0.455, samples=100, color=black] {0.71-2*x^2};
               
\addplot[domain=0.45:0.491, samples=100, color=black]{0.3-50*(x-0.45)^1.6};
               
\addplot[domain=0.45:1, samples=100, color=black]{0.72-2*0.42^1.8-0.1*(x-0.42)+0.14*(x-0.42)^2};

\end{axis}

   		\draw (1.7,1.9) node[] {\scriptsize Confinement};
   		\draw (1.7,1.6) node[] {\scriptsize phase};
   		\draw (5.3,0.9) node[] {\scriptsize Free phase};
   		\draw (4.2,3.7) node[] {\scriptsize Higgs phase};

   		\draw (4.2,3.35) node[] {\tiny \( R\paren{\beta_2,\beta_1,n} \gtrsim 0\)};
   		\draw (1.7,1.25) node[] {\tiny \(  R\paren{\beta_2,\beta_1,n} \gtrsim 0\)};
   		\draw (5.3,0.55) node[] {\tiny \(   R\paren{\beta_2,\beta_1,n} \sim 0\)};

	\end{tikzpicture} 
    \caption{\label{figure: phase diagram} Conjectured limiting behavior of \( R\paren{\beta_2,\beta_1,n} \). We will not address the difference between the Confinement phase and the Higgs phase here. We propose to rigorously establish the diagram in the blue regions.}
	\end{figure}
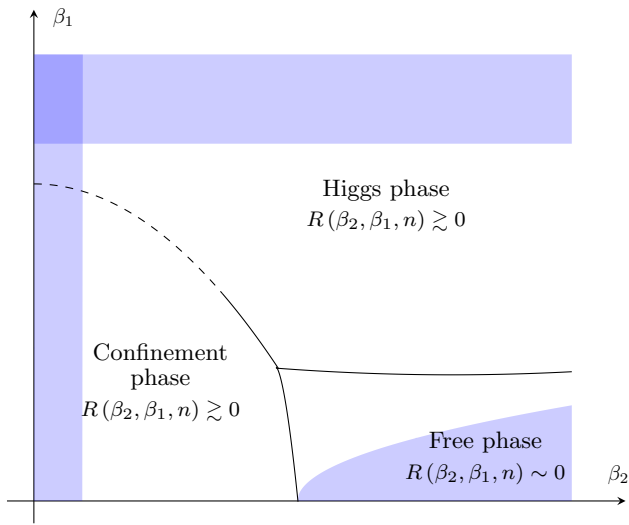
One of the authors recently proved that the Marcu--Fredenhagen ratio exhibits a non-trivial phase transition in the \( \mathbb{Z}_2\) lattice Higgs model~\cite{p2024phase}; we will gain a more detailed understanding for the Potts lattice Higgs model as illustrated by the shaded regions in Figure~\ref{figure: phase diagram}.

\begin{theorem}\label{thm:nontrivial}
    Let $q$ be a prime integer, \( i = 1, \) and $\beta_2,\beta_1\geq 0,$ and let $\mu = \mu_{\beta_2,\beta_1,q,i,\Z^d}.$
    Consider the Marcu--Fredenhagen ratio \( R \) as defined in Definition~\ref{def: MF}. 
    Then the following holds.
\begin{enumerate}
    \item If $\beta_2$ is sufficiently large, then there exists a $\beta_1'$ so that if $\beta_1<\beta_1'$, we have $$ \lim_{n\to \infty} R\paren{\beta_2,\beta_1,n}=0.$$ 
    \item If $\beta_2$ is sufficiently small, then $$\liminf_{n\to\infty}R\paren{\beta_2,\beta_1,n}>0.$$ 
    \item If $\beta_1$ is sufficiently large, then $$\liminf_{n\to\infty}R\paren{\beta_2,\beta_1,n}>0.$$  
\end{enumerate} 
\end{theorem}

In Proposition~\ref{prop: MF for i geq 2}, we show that the natural analogue Marcu--Fredenhagen ratio for \( i \geq 2\) (see Figure~\ref{figure: Wilson lines MF}) does not exhibit a phase transition.

The next few results elucidate the basic properties of the CPP.

\begin{theorem} \label{thm:positive}
$\rho$ is positively associated. That is, if $A$ and $B$ are events that are increasing with respect to $\paren{P_2,P_1}$ then
$$\rho\paren{A\cap B}\geq \rho\paren{A}\rho\paren{B}.$$
\end{theorem} 

We can compare the behavior of the CPP to a pair of independent plaquette percolations in the sense of stochastic domination. 
\begin{definition}[Stochastic Domination]
For a pair of binary assignments $\omega_1,\omega_2$, we say $\omega_1\leq \omega_2$ if for all $z$, $\omega_1(z)=1$ implies $\omega_2(z)=1.$ If there exists a coupling $K(\omega_1,\omega_2)$  of two measures $\pi_1(\omega_1),\pi_2(\omega_2)$ so that $K(\omega_1\leq \omega_2)=1$, we say $\pi_2$ stochastically dominates $\pi_1$ and write $\pi_1\leq_{st}\pi_2$.
\end{definition}

\begin{theorem}
\label{thm:stochdom}
    Let $X$ be a finite cell complex, $q\geq 1,$ and $p_2,p_1\in \brac{0,1}.$ Then $\rho_{p_2,p_1}$ is stochastically dominated by independent Bernoulli percolation with probabilities $p_2$ on $(i+1)$-cells and $p_1$ on $i$-cells, and stochastically dominates percolation with probabilities $\frac{p_2}{q(1-p_2)+p_2}$ and $\frac{p_1}{q(1-p_1)+p_1}$ respectively. Furthermore, when $q$ is fixed, $\rho$ is stochastically increasing in $p_2$ and in $p_1$.   
\end{theorem}

In the text, we prove a stronger version of Theorem~\ref{thm:stochdom}, which concerns a modified CPP where the cohomology coefficients are no longer determined by the parameter $q.$

For the next result, we specialize the the $d$-dimensional torus; specifically, let $\T=\mathbb{T}_d^N$ be the discrete torus of length $N$ obtained by identifying opposite faces of the cube $\brac{0,N}^d\subset\Z^d.$ 

\begin{definition}[Bullet Dual]
    If $\T$ is the $d$-dimensional unit cubical lattice on the torus, we let $(\T)^\bullet$ denote same lattice, shifted by $1/2$ in every coordinate. For every $i$-cell $x\in \T$, there is a unique $(d-i)$-cell $x'$ in $(\T)^\bullet$. For a set $X$, we define $X^\bullet \coloneqq \{x'\in (\T)^\bullet \mid x\notin X\}$.
\end{definition}

\begin{theorem} \label{thm:duality}
    Let $q$ be a prime integer and let $p_2,p_1\in\brac{0,1}.$ Then
    $$
        \rho{\paren{P_2,P_1}} = \rho^{\bullet}{\paren{P_1^{\bullet},P_2^{\bullet}}},
    $$
    where $p_2^{\bullet}=\frac{q(1-p_1)}{p_1+q(1-p_1)}$, $p_1^{\bullet}=\frac{q(1-p_2)}{p_2+q(1-p_2)}$, $\rho=\rho_{p_2,P_1,q,i,d,\T_N^d},$ and $\rho^{\bullet}=\rho_{p_2^{\bullet},p_1^{\bullet},q,d-i-1,(\T)^\bullet}$.
\end{theorem}

Analogues of Theorem~\ref{thm:duality} result can be proven with more general boundary conditions. For example, it is easy to modify the proof to show that the dual of the model with free boundary conditions has wired boundary conditions, when defined appropriately (see, e.g.,~\cite{prcm-24}).

When $d-i-1=i$, the Potts lattice Higgs model possesses a self-dual line in the 2-dimensional parameter space. While we do not do so here, it would be interesting to consider how the coupling in Theorem~\ref{thm:coupling} behaves under the duality transformation. This would yield a further generalization of the Loop--Cluster coupling~\cite{zhang2020loop,hansen2025general}.

We outline the paper. Section~\ref{sec:topos} reviews relevant topological background. Theorems~\ref{thm:coupling} and~\ref{thm:W=V} are proven in Section~\ref{sec:coupling}. This section also covers a coupling for the Potts lattice Higgs~model in general gauge and the equivalence of the CPP with the ghost vertex construction. Sections~\ref{sec:positive}, \ref{sec:stochdom}, and~\ref{sec:duality} include proofs of Theorems~\ref{thm:positive}, \ref{thm:stochdom}, and~\ref{thm:duality}, respectively. 
In Section~\ref{sec:applications}, we derive new and simple proofs of several known properties of the Potts lattice Higgs Model: Griffith's Second Inequality as well as a perimeter law and monotonicity for Wilson line observables.
Finally,  Theorem~\ref{thm:nontrivial} is proven in Section~\ref{sec:ratio}.

\section{Topological Terminology and Techniques} \label{sec:topos}

Throughout the paper, we rely heavily on concepts from algebraic topology. Standard references for these topics are~\cite{H02,bredon2013topology}, while~\cite{kaczynski2004computational,saveliev2016topology} provide introductions specific to cubical complexes.

\subsection{Absolute Homology and Cohomology}
The boundary and coboundary maps defined in Section \ref{sec:defs} have certain useful properties. Composing two boundary maps, or two coboundary maps, always gives the $0$ map: $\partial_j\circ\partial_{j+1}=0, \; \delta_j\circ\delta_{j-1}=0 $. Because $\im \partial_{j+1}\subset \ker \partial_{j} $ (also notated $B_j(X;\,\Z_q)\subset Z_j(X;\,\Z_q)$) and \( \im \partial_{j+1}\) and \( \ker \partial_{j} \) are both linear subspaces, we can define the $j$-dimensional (absolute) \emph{homology group}, 
\[
    H_j(X;\,\Z_q)=\ker \partial_{j} /\im \partial_{j+1} .
\]
The cardinality of \( H_j(X ;\, \Z_q)\) can very roughly be interpreted as the the number of ``linearly independent holes" in the set: the kernel of the boundary map are chains with no boundary, ``\emph{cycles}", while the image are \emph{boundaries}, and the quotient corresponds to cycles that aren't boundaries. Quotienting by boundaries not only removes cycles that are boundaries, but also removes the distinction between cycles whose only difference is a boundary, like two nearby paths encircling the same lake.

Consider the cell complex $X$ shown in Figure~\ref{fig:homologyexample}. It consists of six vertices, seven oriented edges, and one face. We may orient the face $f_1$ so that $\partial f_1=e_1+e_2+e_3+e_4.$ The matrices of the boundary maps \( \partial_1 \) and \( \partial_2 \) with respect to the ordered bases $\set{v_1,\ldots, v_6}$ of $C_0(X;\,\Z_q),$  $\set{e_1,\ldots, e_7}$ of $C_1(X;\,\Z_q),$ and $\set{f_1}$ for $C_2(X;\,\Z_q)$ are 

$$\partial_1= \begin{bmatrix}
-1 & 1 & 0 & 0 & 0 & 0 & 0 \\
1 & 0 & 0 & -1 & 0 & 0 & 0 \\
0 & 0 & -1 & 1 & 0 & 0 & -1 \\
0 & -1 & 1 & 0 & 1 & 0 & 0 \\
0 & 0 & 0 & 0 & -1 & 1 & 0 \\
0 & 0 & 0 & 0 & 0 & -1 & 1 
\end{bmatrix} \qquad  \text{and} \qquad \partial_2=\begin{bmatrix}
1 \\
1 \\
1 \\
1 \\
0 \\
0 \\
0 
\end{bmatrix}.$$
An elementary computation yields that 
$$
    Z_1(X;\,\Z_q)=\ker\partial_1=\mathrm{Span}\paren{e_1+e_2+e_3+e_4,e_3-e_5-e_6-e_7}
$$ 
and
$$
    B_1(X;\,\Z_q)=\im\partial_2=\mathrm{Span}\paren{e_1+e_2+e_3+e_4}.
$$ 
Consequently, a basis for $H_1(X;\,\Z_q)=Z_1(X;\,\Z_q)/B_1(X;\,\Z_q)$ is $\set{e_3-e_4-e_6-e_7}$, and thus $H_1(X;\,\Z_q)\cong \Z_q.$

\begin{figure}[t]
    \centering

\begin{tikzpicture}[scale=1.8, decoration={
    markings,
    mark=at position 0.5 with {\arrow{>}}}
    ]

    \fill[fill=blue, fill opacity = .2] (0,0) -- (2,0) -- (2,2) -- (0,2) -- cycle;

    \draw (1,1) node {\( f_1\)};
    
    \draw[thick,postaction={decorate}] (0,0) -- (2,0) node[midway,below,yshift=-1ex] {\( e_4\)}; 
    \draw[thick,postaction={decorate}] (2,0) -- (4,0) node[midway,below,yshift=-1ex] {\( e_7\)}; 
    \draw[thick,postaction={decorate}] (4,0) -- (4,2) node[midway,right,xshift=.5ex] {\( e_6\)}; 
    \draw[thick,postaction={decorate}] (4,2) -- (2,2) node[midway,above,yshift=1ex] {\( e_5\)}; 
    \draw[thick,postaction={decorate}] (2,2) -- (0,2) node[midway,above,yshift=1ex] {\( e_2\)}; 
    \draw[thick,postaction={decorate}] (0,2) -- (0,0) node[midway,left,xshift=-1ex] {\( e_1\)};  
    \draw[thick,postaction={decorate}] (2,0) -- (2,2) node[midway,left,xshift=-1ex] {\( e_3\)};  
    
    \fill (0,0) circle (1.6pt) node[below,yshift=-1ex] {\( v_2\)}; 
    \fill (2,0) circle (1.6pt) node[below,yshift=-1ex] {\( v_3\)}; 
    \fill (4,0) circle (1.6pt) node[below,yshift=-1ex] {\( v_6\)}; 
    \fill (0,2) circle (1.6pt) node[above,yshift=1ex] {\( v_1\)}; 
    \fill (2,2) circle (1.6pt) node[above,yshift=1ex] {\( v_4\)}; 
    \fill (4,2) circle (1.6pt) node[above,yshift=1ex] {\( v_5\)}; 
\end{tikzpicture}
    \caption{A cell complex with six vertices, seven edges, and one face.} 
    
    \label{fig:homologyexample}
\end{figure}
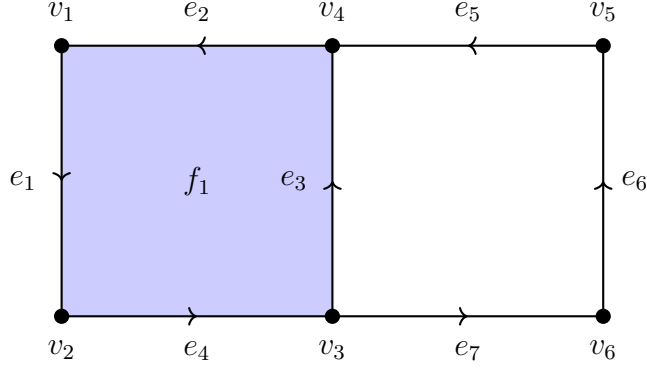

We now define the \emph{cocycles} $Z^j(X;\,\Z_q)$, the \emph{coboundaries} $B^j(X;\,\Z_q)$, and the \emph{cohomology} \(  H^j(X;\,\Z_q), \) by
\begin{equation*}
    Z_j(X;\,\Z_q)=\ker\delta_j, \quad  B_j(X;\,\Z_q)=\im\delta_{j-1} , \text{ and }  H^j(X;\,\Z_q)=\ker \delta_{j} /\im \delta_{j-1} . 
\end{equation*}
Then \( Z_j(X;\,\Z_q)\) consists of functions $f\in C^j(X;\,\Z_q)$ such that $f(\partial\sigma)=0$ for all $(j+1)$-cells $\sigma$ in \( X\), \( B^j(X;\,\Z_q)\) consists of functions which can be written as $\partial \phi$ for some $\phi\in C^{j+1}(X;\,\Z_q)$, and the cohomology $H^j(X;\,\Z_q)$ consists of cocycles modulo coboundaries. The cohomology is more difficult to explain intuitively, but can be thought of as the ``obstruction to discrete integrability" for the space $X.$ Specializing to the case where $P_2$ is an $(i+1)$-dimensional percolation subcomplex, $Z^1\paren{P_2;\,\Z_q}$ is the collection of spin assignments compatible with the pair $\paren{\varnothing,P_2}$ and $B^1\paren{P_2;\,\Z_q}$ is the group of gauge symmetries.

We now return to the example in Figure~\ref{fig:homologyexample}. Choosing bases $\set{v_1^*,\ldots,v_6^*}$ for  $C^0\paren{X}$, $\set{e_1^*,\ldots,e_7^*}$ for $C^1(X;\,\Z_q),$
and $\set{f_1^*}$ for $C^2(X;\,\Z_q)$, dual to the bases given above (where, for an $i$-cell $\sigma$ the cochain $\sigma^*$ is defined by $\sigma^*\paren{\sigma'}
=I_{\sigma=\sigma'}$), we obtain matrices for the coboundary maps $\delta_0=\partial_1^T$ and $\delta_1=\partial_2^T.$ In particular, it follows that 
\begin{align*}
Z^1(X;\,\Z_q)&=\mathrm{Span}\paren{e_1^*-e_2^*,e_1^*-e_3^*,e_1^*-e_4^*,e_5^*,e_6^*,e_7^*}
\end{align*}
and
\begin{align*}
B^1(X;\,\Z_q)&=\mathrm{Span}\paren{-e_1^*+e_2^*,e_1^*-e_4^*,-e_3^*+e_4^*-e_7^*,-e_2^*+e_3^*+e_5^*,-e_5^*+e_6^*},
\end{align*}
and hence $H^1(X;\,\Z_q)\cong \Z_q$ is a vector space of dimension one. There are many choices of basis for this cohomology group, for example one could choose $e_3^*-e_4^*-e_6^*-e_7^*.$ One could ask whether it is always true that $H^1(X;\,\Z_q)\cong H_1(X;\,\Z_q)$; this fails when $q$ is non-prime and the simplest examples are embeddable in $\R^4$ but not in $\R^3.$  Here, we fix $q$ to be prime so $\Z_q$ is a field, and the Universal Coefficient Theorem for Cohomology (Theorem 3.2 in~\cite{H02}) implies that in this case,
$$
    H^j(X,A;\,\Z_q)\cong 
    H_j(X,A;\,\Z_q)
$$
for all $j$. One might be tempted to believe that the dimensions of the vector spaces $H^j(X;\,\Z_q)$ and  $H_j(X;\,\Z_q)$ do not depend on $q.$ However, while this is true for $j=0,$ it is false for $j>0.$ The simplest examples for which this fails are non-orientable surfaces such as the real projective plane and Klein bottle.  When we are not relying on particular properties of the coefficient field, we usually suppress it for brevity and write, e.g., \( C^j(X) \) instead of \( C^j(X;\,\Z_q),\) but we will note whenever it is relevant.

\subsection{Relative Homology and Cohomology}

Homology and cohomology apply to a single complex $X$. However, we can perform a similar process for a pair of complexes $(X,A)$, where $A$ is a subcomplex of $X$. For so-called ``good pairs'' $\paren{X,A},$ the \emph{relative homology} $H_j\paren{X,A}$ and \emph{relative cohomology} $H^j\paren{X,A}$ are respectively isomorphic to the absolute homology $H_j\paren{X/A}$ and absolute cohomology $H^j\paren{X/A}$ of the quotient space $X/A$ formed by collapsing $A$ to a single point (where we are suppressing the dependence on the choice of coefficients). See Theorem 2.13 in~\cite{H02} for the statement of this result for homology and the discussion below for the definition of a good pair. The proof for cohomology proceeds identically. While all pairs considered in this paper are deterministically good, it will be useful to provide an equivalent definition via relative chain and cochain groups. 

The relative chain group is the quotient vector space $C_j(X,A)=C_j(X)/C_j(A)$. As the boundary of a chain supported on $A$ is also supported on $A,$ the usual boundary map induces a relative boundary map $\partial_j: C_j(X,A) \to C_{j-1}(X,A)$, and we can define relative homology $H_j(X,A)=\ker(\partial_{j})/\im(\partial_{j+1}).$  On the other hand, the relative cochain group  $C^j(X,A)$ is the subspace of $C^j(X)$ consisting of cochains that vanish on all $j$-cells of $A$ and the relative coboundary map $\delta_j$ is the restriction of the absolute coboundary map to $C^j(X,A).$ Again, $H^j(X,A)=\ker \delta_{j} /\im \delta_{j-1} .$ 
In our setting, the relative cohomology is much simpler. If $P_2$ and $P_1$ are a pair of percolation subcomplexes of dimensions $(i+1)$ and $i,$ then they share the same $(i-1)$-cells so  $C^{i-1}(P_2,P_1)=0. $ Thus $H^{i}(P_2,P_1)=\ker\paren{\delta^i\paren{C^i(P_2,P_1)}}\coloneqq Z^i\paren{P_2,P_1}$, which is simply the set of relative co-cycles: cochains compatible with $\paren{P_2,P_1}$ in the sense of Definition~\ref{defn:compatible}.

Returning to the example $X$ in Figure~\ref{fig:homologyexample}, if we set $A$ to contain all cells of $X$ except $e_5,$ $e_6,$ and $e_7$ then $C^0\paren{X,A}=\set{0},$ $C^1(X,A)=\mathrm{Span}\paren{e_5^*,e_6^*,e_7^*},$ and  $C^2(X,A)=\set{0}.$ It follows that $H^1(X,A)= C^1(X,A)\cong \Z_q^3.$ Alternatively, the quotient space $X/A$ consists of three circles meeting at a point, and it is easily seen that $H^1(X/A)\cong H_1(X/A)\cong \Z_q^3.$ 


Since $q$ is a prime number, $H^j\paren{X,A;\,\Z_q}$ and $H_j\paren{X,A;\,\Z_q}$ are vector spaces (as quotients of vector spaces) and are determined by their dimensions.  
\begin{definition}[Betti numbers]\label{defn:betti} 
    The $j$-dimensional \emph{relative Betti number} is 
    $$
        b_j(X,A;\,\Z_q) \coloneqq \mathrm{rank}\bigl(H^j(X,A;\,\Z_q)\bigr)
        =
        \mathrm{rank}\bigl(H_j(X,A;\,\Z_q)\bigr).
    $$
\end{definition}

By applying the rank-nullity theorem to the boundary maps, we obtain the famous \emph{Euler--Poincar\'{e} formula}.
\begin{prop}\label{prop:euler}
If $X$ is a finite, $d$-dimensional cell complex and $q$ is a prime number then 
\[\chi\paren{X}\coloneqq \sum_{j=0}^{d}(-1))^{j} |X^{\brac{j}}|=\paren{-1}^{j}b_j(X,A;\,\Z_q).\]
\end{prop}

\subsection{Mayer--Vietoris Sequences}\label{sec:MV}
In this section, we describe the Mayer--Vietoris sequence for cohomology, a key technical tool that relates the relative cohomology of a pair of spaces with that of their union and their intersection~\cite[pp 149--153 and 203--204]{H02}. We will apply the Mayer--Vietoris sequence to prove that the CPP is positively associated in Section~\ref{sec:positive} and to demonstrate a technical lemma in Section~\ref{sec:ratio}. A reader may wish to skip this section for now.

While we apply the Mayer--Vietoris sequence for relative cohomology below, we begin by describing the sequence for absolute cohomology for the purpose of simplicity. To this end, note first that if \( P \) and \( Q\) are two cell complexes with $P\subset Q$, then the inclusion map $P\hookrightarrow Q$ induces a restriction map $H^j\paren{Q}\to H^{j}\paren{P}$ that sends the homology class $\brac{f}$ of a cocycle $f$ supported on $Q$ to the cohomology class $\brac{f|_{P}}$ of its restriction to $P.$ 
Now, suppose that $P$ and $Q$ are subcomplexes of the same cell complex. We obtain four different restriction maps on cohomology $\phi_{P}: H^j\paren{P}\to H^j\paren{P\cap Q},$  $\phi_{Q}: H^j\paren{Q}\to H^j\paren{P\cap Q},$ $\psi_{P}: H^j\paren{P\cup Q}\to H^j\paren{P},$ and  $\psi_{Q}: H^j\paren{P\cup Q}\to H^j\paren{Q}.$ These maps fit together in a sequence
\begin{equation}\label{eq: mv diagram 1}
    H^j(P\cap Q)  
    \mvleftarrow{\phi_P-\phi_Q}
    H^j(P) \oplus H^j(Q) 
    \mvleftarrow{\psi_P \oplus \psi_Q}
    H^j(P\cup Q). 
\end{equation}
with the following properties.
\begin{enumerate}
    \item If $\brac{f}\oplus\brac{g}\in \im (\psi_P\oplus\psi_Q)$, then $\phi_P\paren{\brac{f}}-\phi_{Q}\paren{\brac{g}}=0.$
    \item In fact, the converse implication also holds: if  $\phi_P\paren{\brac{f}}-\phi_{Q}\paren{\brac{g}}=\brac{0}$, then there is an $h\in C^{j-1}\paren{P\cap Q}$ so that $\delta h +f|_{P\cap Q} =g|_{P\cap Q}.$ The map $h$ can be extended by zero to define a cochain on all of $P$ so we may replace $f$ with $f+\delta h$ if necessary to find a representative of the same cohomology class of $P$ so that $f|_{P\cap Q}=g|_{P\cap Q}.$ Then we can define $F\in C^i\paren{P\cup Q}$ by $F\paren{\sigma}=f\paren{\sigma}$ if $\sigma \in P$ and $F\paren{\sigma}=g\paren{\sigma}$ otherwise. Since $f$ and $g$ are cocycles on $P$ and $Q,$ it follows that $F$ is a cocycle supported on $P\cup Q,$  and hence \( [f] \oplus [g] \in \im (\psi_x \oplus \psi_Q).\)
\end{enumerate}
Combining the above observations, it follows that the sequence in~\ref{eq: mv diagram 1} is \emph{exact}, i.e., that
$$
    \psi_P\oplus\psi_Q(F)=f\oplus g \quad \text{and} \quad \im \psi_P\oplus\psi_Q = \ker \paren{\phi_P- \phi_Q}.
$$
This is not quite enough to compute the cohomology groups of $P\cup Q$ in terms of those of $P,$ $Q$, and $P\cap Q$; we also require knowledge of the image of $\phi_P- \phi_Q$ and the kernel of $\psi_P\oplus\psi_Q.$ It turns out that the coboundary map provides the required information. Specifically, it induces a map $\delta^*:H^j\paren{P\cap Q}\to H^{j+1}\paren{P\cap Q}$ so that the following sequence is exact in the sense that the image of each map is the kernel of the next:

\begin{center}
    \begin{tikzcd}
            \cdots & H^j(P\cap Q) \arrow[l, "{\delta^*}",swap] & H^j(P) \oplus H^j(Q) \arrow[l,"{\phi_P- \phi_Q}",swap] & H^j(P\cup Q) \arrow[l,"{\psi_P\oplus\psi_Q}",swap]  & \hphantom{0} \arrow[l, "{\delta^*}",swap]\\
               \hphantom{\cdots} & H^{j-1}(P\cap Q) \arrow[l, "{\delta^*}",swap]
        & \makebox[\widthof{$H^j(P) \oplus H^j(Q)$}][c]{$\cdots\hfill \cdots$}\arrow[l,"{\phi_P- \phi_Q}",swap]
        &  H^0(P\cup Q) \arrow[l,"{\psi_P\oplus\psi_Q}",swap] & 0 \lar.
    \end{tikzcd}
\end{center}
We will not use the precise construction of $\delta^*$ here, only that the sequence is exact.


This construction extends to relative cohomology. Assume that $A$ and $B$ are subcomplexes of $P$ and $Q,$ respectively. Then, the Mayer--Vietoris sequence for relative cohomology is as follows.
\begin{center}
    \begin{tikzcd}[arrows=to,cramped, sep = small]
            \cdots & H^j(P\cap Y, A \cap B) \lar & H^j(P,A) \oplus H^j(Y,B) \lar & H^j(P\cup Y, A\cup B) \lar  & \hphantom{0}\lar\\
               \hphantom{\cdots} & H^{j-1}(P\cap Y, A \cap B) \lar 
        & \makebox[\widthof{$H^j(P,A) \oplus H^j(Y,B)$}][c]{$\cdots\hfill \cdots$} \lar
        &  H^0(P\cup Y, A\cup B) \lar & 0 \lar.
    \end{tikzcd}
\end{center}



\subsection{Alexander Duality}

Alexander duality relates the homology of a subset of Euclidean space with the cohomology of its complement. It generalizes the topological fact that the number of linearly independent loops of a bounded subset $X$ of $\R^2$ equals the number of bounded components of the complement $\R^2\smallsetminus X$. If $A\subset X$ is contained in an orientable compact $d$-dimensional manifold $S$, where $A$ is closed and $X$ is compact, this can be further generalized to relate the relative homology of the pair $(X,A)$ with the relative cohomology of  $(S\smallsetminus A, S\smallsetminus X)$~\cite[Theorem 6.2.17]{S95}:
$$
\label{eqn:alexander}
    H^{j}(S\smallsetminus A,S\smallsetminus X;\,\Z_q)\cong H_{d-j}(X,A;\,\Z_q).
$$ 
For example, if $X$ is a compact subset of $\R^2,$ and $A\subset X$ contains a single non-contractible loop then the rank of $H_1(X,A)$ will be one less than the rank of $H_1\paren{X}$, and $H^1(\R^2\smallsetminus A,\R^2\smallsetminus X)$ will also have rank one less than $H^1(\R^2\smallsetminus A,\R^2\smallsetminus X).$ 

\section{The coupling between the CPP and the Potts lattice Higgs model} \label{sec:coupling}

In this section, we prove Theorem~\ref{thm:coupling} and discuss some of its consequences.

\subsection{Derivation of the Coupling}

For clarity of notation: $\sigma$ always refers to an $(i+1)$-cell, $\epsilon$ always refers to an $i$-cell, $j$-cells contained within a given $j$-dimensional percolation subcomplex are called ``open", and are otherwise called ``closed.''

\begin{proof}[\textbf{Proof of Theorem~\ref{thm:coupling}}] \label{pf:coupling}

We first compute the marginals of \( \kappa\paren{f,P_2,P_1}.\) To this end, note first that
\begin{align*}
    \kappa_1(f)&\coloneqq\sum_{P_1}\sum_{P_2}\kappa\paren{f,P_2,P_1}
    \\&\propto
    \sum_{P_1}\sum_{P_2}\Bigl[\prod_{\epsilon\in X^{\brac{i}}}\pigl[I_0\bigl(P_1(\epsilon)\bigr)+k_1 P_1(\epsilon)I_0\bigl(f(\epsilon)\bigr)\pigr]
    \prod_{\sigma\in {X^{\brac{i+1}}}}\pigl[I_0\bigl(P_2(\sigma)\bigr)+k_2 P_2(\sigma)I_0 \bigl(\delta f(\sigma) \bigr) \pigr]\Bigr].
\end{align*}
Changing the order of the sums and products, it follows that the previous expression is equal to
\begin{align*}
    &\prod_{\epsilon\in X^{[i]}}\pigl[1+k_1 I_0\bigl(f(\epsilon)\bigr)\pigr]\prod_{\sigma\in X^{[i+1]}} \pigl[1+k_2 I_0\bigl(\delta f(\sigma)\bigr)\pigr]
    \\
    &\qquad =\prod_{\epsilon\in X^{[i]}}\pigl[1+(e^{\beta_1}-1)I_0\bigl(f(\epsilon)\bigr)\pigr]\prod_{\sigma\in X^{[i+1]}} \pigl[1+(e^{\beta_2}-1)I_0\bigl(\delta f(\sigma)\bigr)\pigr]\\
    &\qquad =\prod_{\epsilon\in {X^{\brac{i}}}}e^{{\beta_1} I_0(f(\sigma))}\prod_{\sigma\in {X^{\brac{i+1}}}}e^{{\beta_2} I_0(\delta f(\sigma))}.
\end{align*}
Next, for the second marginal, we have
\begin{align}\label{eqn:kappa2}
\begin{split}
     \kappa_2(P_2,P_1)&\coloneqq\sum_{f}\kappa\paren{f,P_2,P_1}\\
   &\propto \sum_f\Bigl[  \prod_{\epsilon\in {X^{\brac{i}}}}\pigl[I_0\pigl(P_1(\epsilon)\bigr)+k_1 P_1(\epsilon)I_0\bigl(f(\epsilon)\bigr)\pigr]\prod_{\sigma\in {X^{\brac{i+1}}}}\pigl[I_0\bigl(P_2(\sigma)\bigr)+k_2 P_2(\sigma)I_0\bigl(\delta f(\sigma)\bigr)\pigr]\Bigr]\\
    &={k_1}^{|P_1|}{k_2}^{|P_2|} \sum_f \prod_{\epsilon \colon P_1(\epsilon)=1}I_0\bigl(f(\sigma)\bigr)\prod_{\sigma \colon P_2(\sigma)=1}I_0\bigl(\delta f(\sigma)\bigr)\\
    &={k_1}^{|P_1|}{k_2}^{|P_2|}\bigl|\bigl\{f \colon \forall\epsilon\in P_1, f(\epsilon)=0,\, \forall\sigma\in P_2, \delta f(\sigma)=0\bigr\}\bigr|\\
    &\propto (1-p_1)^{|{X^{\brac{i}}}|-|P_1|}p_1^{|P_1|}(1-p_2)^{|{X^{\brac{i+1}}}|-|P_2|}p_2^{|P_2|}\\ &\qquad\cdot\bigl|\bigl\{f \colon \forall\epsilon\in P_1, f(\epsilon)=0,\, \forall\sigma\in P_2, \delta f(\sigma)=0\bigr\}\bigr|.
    \end{split}
\end{align}
Since
\[
Z^{i}(P_2,P_1) = \{f \colon \forall\epsilon, f(\epsilon)=0,\, \forall\sigma, \delta f(\sigma)=0\},
\]
and, as we mentioned in Section \ref{sec:topos}, we have $H^{i}(P_2,P_2)=Z^{i}(P_2,P_1)$, it follows that 
\begin{align*}
    \kappa_2(P_2,P_1) \propto (1-p_1)^{|{X^{\brac{i}}}|-|P_1|}p_1^{|P_1|}(1-p_2)^{|{X^{\brac{i+1}}}|-|P_2|}p_2^{|P_2|}|H^i(P_2,P_1;\,\Z_q)|.
\end{align*}

Now we derive the conditional distributions of \(\kappa(f,P_2,P_1)\). To this end, note first that
\begin{align*}
    \kappa(f \mid P_2,P_1)&\propto \prod_{\epsilon\in {X^{\brac{i}}}}\pigl[I_0\bigl(P_1(\epsilon)\bigr)+k_1 P_1(\epsilon)I_0\bigl(f(\epsilon)\bigr)\pigr]
    \prod_{\sigma\in {X^{\brac{i+1}}}}\pigl[I_0\bigl(P_2(\sigma)\bigr)+k_2 P_2(\sigma)I_0\bigl(\delta f(\sigma)\bigl)\pigr] 
    \\ &= \prod_{\epsilon \colon P_1(\epsilon)=1}hI_0(f(\epsilon))\prod_{\sigma P_2(\sigma)=1}KI_0(\delta f(\sigma)).
\end{align*}
In other words, given \( P_1 \) and \(P_2 \), if either $\delta f(\sigma)\neq 0$ for some $\sigma\in P_2$ or  $f(\epsilon)\neq 0$ for some $\epsilon \in P_1$, then \( \kappa( f\mid P_2,P_1)=0, \) and otherwise, the distribution of \( \kappa( \cdot \mid P_2,P_1)\) is uniform.
For the other conditional marginal distribution, we have
\begin{align*}
    \kappa(P_2,P_1 \mid f)\propto& \prod_{\epsilon\in {X^{\brac{i}}}}\Bigl[ I_0\bigl(P_1(\epsilon)\bigr)+k_1 P_1(\epsilon)I_0\bigl(f(\epsilon)\bigr) \Bigr]
    \\&\qquad\cdot \prod_{\sigma\in {X^{\brac{i+1}}}}\Bigl[I_0\bigl(P_2(\sigma)\bigr)+k_2 P_2(\sigma)I_0\bigl(\delta f(\sigma)\bigr) \Bigr] 
    \\
    = 
    &\prod_{\epsilon \colon f(\epsilon)=0}\Bigl[I_0\bigl(P_1(\epsilon)\bigr)+k_1 P_1(\epsilon) \Bigr]\prod_{\epsilon \colon f(\epsilon)\neq 0}I_0\bigl(P_1(\epsilon)\bigr)
    \\&\qquad\cdot\prod_{\sigma \colon \delta f(\sigma)=0}\Bigl[I_0\bigl(P_2(\sigma)\bigr)+k_2 P_2(\sigma) \Bigr]\prod_{\sigma \colon \delta f(\sigma)\neq 0}I_0\bigl(P_2(\sigma)\bigr) \\
    \propto &\prod_{\epsilon \colon f(\epsilon)=0} \Bigl[(1-p_1)I_0\bigl(P_1(\epsilon)\bigr)+p_1P_1(\epsilon)\Bigr] \prod_{\epsilon \colon f(\epsilon)\neq 0}I_0\bigl(P_1(\epsilon)\bigr) \\
    & \qquad\cdot \prod_{\sigma \colon \delta f(\sigma)=0} \Bigl[(1-p_2)I_0 \bigl(P_2(\sigma)\bigr)+p_2P_2(\sigma)\Bigr] \prod_{\sigma \colon\delta f(\sigma)\neq 0}I_0\bigl(P_2(\sigma)\bigr).
\end{align*}
In words, the $i$-cells on which $f$ vanishes and the $(i+1)$ cells on which $\delta f$ is zero are included independently. This concludes the proof.

\end{proof}

\subsection{Topological Interpretation of Wilson Line Observables}

In this section, we apply the coupling from Theorem~\ref{thm:coupling} to express Wilson line expectations to be expressed in terms of the probability of a topological event in the CPP. 

\begin{proof}[\textbf{Proof of Theorem \ref{thm:W=V}}]
    It suffices to show that $\E_\kappa(W_\gamma\mid V_\gamma)=1$ and $\E_\kappa(W_\gamma\mid V_\gamma^c)=0.$  To this end, fix \(\gamma\),  \(P_2\), and \(P_1,\) and let \( f \sim \kappa(\cdot\mid P_2,P_1). \) 

    If $[\gamma]=0$, then $\gamma=\sigma+\partial \tau$ for some $\sigma\in C_i(P_1)$ and $\tau\in C_{i+1}(P_2)$. Since $f$ is sampled uniformly from $Z^{i}(P_2,P_1;\,\Z_q)$, we have 
    $$
        f(\gamma)=f(\sigma)+f(\partial \tau)=f(\sigma)+\delta f(\tau)=0+0=0,
    $$
    where we used that $f$ vanishes on $P_1$ and $\delta f$ vanishes on $P_2.$ Thus $$W_\gamma = \exp \bigl( \frac{2\pi i}{q} f(\gamma)\bigr)= 1.$$

The conditional distribution of $f$ given $\paren{P_2,P_1}$ can be sampled by fixing a basis $\set{f_1,\ldots,f_k}$ of the vector space $H^{i}(P_2,P_1;\,\Z_q),$ letting $a_1,\ldots,a_k$ be distributed independently and uniformly on $\Z_q,$ and setting
$$f=a_1f_1+\ldots a_kf_k.$$
Now assume $[\gamma]\neq0$. By the Universal Coefficient Theorem there exists a dual $[\gamma]^*\in H^{i}(P_2,P_1;\,\Z_q)$ so that $[\gamma]^*([\gamma])=1.$ In particular, $f_j\paren{\brac{\gamma}}\neq0$ for at least one index $j$. By multiplying each basis element by a non-zero element of $\Z_q$ we may assume, without loss of generality, that $f_j(\gamma) =1$ for $1\leq j\leq m$ and $f_j(\gamma)=0$ otherwise. Then
$$f(\gamma)=a_1+\ldots+a_m$$
is a sum of i.i.d. uniform elements of $\Z_q,$ and is itself uniformly distributed on $\Z_q.$ This implies in particular that $\E_\kappa(W_\gamma\mid V_\gamma^c)=0,$ which is the desired conclusion.


\end{proof}

\subsection{Potts Lattice Higgs in General Gauge}
\label{sec:generalgauge}
The proof of Theorem~\ref{thm:coupling} can be modified in a straightforward fashion to produce a coupling between the CPP and the Potts lattice Higgs model in general gauge. Recall that $\nu=\nu_{\beta_2,\beta_1,q,i}$ is the measure induced by the Hamiltonain in~\eqref{eq:generalgauge}. 

\begin{corollary}
    Under the assumptions of Theorem~\ref{thm:coupling}, there is a coupling $\hat{\kappa}=\hat{\kappa}_{k_2,k_1,i,q}$, with $\hat{\kappa}\colon (C^i(X;\,\Z_q)\times C^{i-1}(X;\,\Z_q))\times (\Xi^{i+1}\times \Xi^{i})\to \brac{0,1},$ defined by
    \begin{align*}
        \hat{\kappa}((f,g),(P_2,P_1))\propto &\prod_{\epsilon\in X^{\brac{i}}} \Bigl[I_0\bigl(P_1(\epsilon)\bigr)+k_1\cdot P_1(\epsilon)\cdot I_{\delta g(\epsilon)}\bigl(f(\epsilon) \bigr) \Bigr]
        \\&\qquad\cdot
        \prod_{\sigma\in X^{\brac{i+1}}}\Bigl[I_0\bigl(P_2(\sigma)\bigr)+k_2\cdot P_2(\sigma)\cdot I_0\bigl(\delta f(\sigma)\bigr)\Bigr]
    \end{align*}
    which satisfies the following. 
    \begin{itemize}
        \item The marginal distribution on $\paren{f,g}$ is $\nu_{\beta_2,\beta_1,q,i,X}$ (the Potts lattice Higgs model in general gauge).
        \item The marginal distribution on $\paren{P_2,P_1}$ is $\rho_{p_2,p_1,q,i,X}$ (the CPP). 
        \item The conditional distribution on $\paren{P_2,P_1}$ given $(f,g)$  is, respectively, independent percolation with probability $p_2$ on the set of $(i+1)$-cells so that $\delta f \paren{\sigma}=0$ and probability $p_1$ on the set of $i$-cells so that ${f\paren{\epsilon}=\delta g(\epsilon)}.$   
        \item The conditional distribution of $f$ given $\paren{P_2,P_1}$ and $g$ is the uniform distribution on the set $\delta g+ Z^i\paren{P_2,P_1;\,\Z_q}$  and the conditional distribution of $g$ given $\paren{P_2,P_1}$ is uniform on $ C^{i-1}(X;\,\Z_q)).$ 
    \end{itemize}
\end{corollary}

Note that while $\delta g$ is necessarily an element of $Z^i\paren{P_2;\,\Z_q},$ it may not be a member of of $Z^i\paren{P_2,P_1;\,\Z_q}.$ 

\begin{proof}
The proof structure is very similar to that of Theorem~\ref{thm:coupling}. We note only the changes. To compute the marginal on $\paren{f,g}$ it suffices to simply replace $I_{0}(f(\epsilon))$ with $I_{\delta g(\epsilon)}(f(\epsilon))$ in the computation of the marginal of $\kappa$ on $f.$ \\*
    For the marginal 
    $$  
        \hat{\kappa}_2 (P_2,P_1)\coloneqq\sum_{f,g}\kappa\bigl((f,g),\paren{P_2,P_1}\bigr)
    $$
    on $(P_2,P_1)$, the proof proceeds identically as in~\eqref{eqn:kappa2} except that the factor 
    $$ 
        \bigl\{ f \mid \forall\epsilon, f(\epsilon)=0,\, \forall\sigma, \delta f(\sigma)=0 \bigr\}
        =\bigl|Z^i\paren{P_2,P_1;\,\Z_q}\bigr|
    $$
    in the final line is replaced with
    \begin{equation*}
        \sum_{g\in C^{i-1}(X)} \big|\{f \mid \forall\epsilon, f(\epsilon)=\delta g(\epsilon),\, \forall\sigma, \delta f(\sigma)=0\}\big|,
    \end{equation*}
which in turn equals $|{C^{i-1}(X)} | \times |Z^i (P_2,P_1;\,\Z_q)|$ since the summand does not depend on $g.$ 
    Thus
    $$
        \hat{\kappa}_2(P_2,P_1)
        =
        |C^{i-1}(X)|
        \kappa_2({P_2,P_1}) 
        \propto \kappa_2({P_2,P_1}) \propto \rho({P_2,P_1}).$$
For the conditional distribution of $(f,g)$ given $(P_2,P_1)$, we again plug in $I_{\delta g(\epsilon)}(f(\epsilon))$, to get that the probability is 0 if for any $\sigma\in P_2$, $\delta f(\sigma)\neq 0$ or any $\epsilon \in P_1$, $f(\epsilon)\neq \delta g(\epsilon)$, and otherwise uniform. Because the number of such $f$ is not dependent on $g$, we can select $g$ first by independent choices from $\Z_q$ on each $(i-1)$-cell, uniformly choose a cocycle $h$, and then define $f$ as $h+\delta g$.\\*
For the conditional of $(P_2,P_1)$ given $(f,g)$, we continue to substitute $I_{\delta g(\epsilon)}(f(\epsilon))$, which results in probability 0 if $P_1(\epsilon)=1$ for any $\epsilon$ so that $f(\epsilon)\neq \delta g (\epsilon)$ or $P_2(\sigma)=1$ for any $\sigma$ so that $df(\sigma)=0$, and is, as above, otherwise independent percolation.
\end{proof}





\subsection{Equivalence with the Ghost Vertex Construction}
\label{sec:ghost}
In this section, we explain why the $i=0$ case of the CPP is equivalent to the graphical representation of the Potts model with external field obtained by introducing a ``ghost vertex'' to the random cluster model. To this end, let $G$ be a graph and denote by $G'$ the graph obtained by adding one additional ``ghost'' vertex $g$ to the vertex set of $G$ and an edge $e_v=\paren{v,g}$ for each vertex $v$ of $G.$ Also, for $f\in C^0\paren{G;\,\Z_q}$, let $f' \in C^0\paren{G';\,\Z_q}$ be the ``extension by $0$'' of $f$; that is $f'\paren{v}=f\paren{v}$ for any vertex $v$ of $G$ and $f'\paren{g}=0$. Moreover, for an edge $e$, let 
$$
    p\paren{e} = \begin{cases}p_2\coloneqq \frac{k_2}{1+k_2} & \text{if } e\in E(G),\text{ and}\\ p_1\coloneqq\frac{k_1}{1+k_1} & \text{if } e\notin E(G). \end{cases}
$$
Define a coupling $\kappa'=\kappa'_{k_2,k_1,q,G},$ $\kappa'\colon C^0\paren{G;\,\Z_q}\times \Xi^1\paren{G'}\to \brac{0,1}$ by
\begin{align*}
    \kappa'\paren{f,P}
    \propto &\prod_{e=(g,v)}\Bigl[ \paren{1-p_1}I_0(P(e)) + p_1 P(e) I_0\bigl(\delta f'(e)\bigr)\Bigr] 
    \\  &\qquad\cdot 
    \prod_{e \in E(G)}\Bigl[\paren{1-p_2}I_0\bigl(P(e)\bigr) + p_2 P(e) I_0\bigl(\delta f'\paren{e}\bigr)\Bigr].
    \end{align*}\
That is, $\kappa'$ is obtained from the usual coupling between the random cluster model on $G'$ (with two different types of edges) and the Potts model on $G'$ conditioned on $f'(g))=0$. 

\begin{prop}
    Assume the hypotheses of Theorem~\ref{thm:coupling} with $X=G.$ Let $P_1\in \Xi^0(G)$ and $P_2\in \Xi^1({G})$ and let $P\in \Xi^1({G'})$ be obtained from $P_2$ by adding all edges of the form $({v,g})$ where $v\in P_1.$ Then 
    $$
        \kappa(f,P_2,P_1)=\kappa'(f',P)
    $$
    where $\kappa'=\kappa'({k_2,k_1,q,G}).$ 
    In particular, $\mu(f)=\mu'(f')$ where $\mu'$ is the usual Potts measure on $G'$ conditioned on $f(g)=0$, and $\rho(P_2,P_1)=\rho'(P')$, where $\rho'$ is the random cluster measure with probability $p_2$ assigned to edges in $G$ and $p_1$ assigned to edges in $G'\setminus G.$ 
\end{prop}

\begin{proof}
First, note that $f$ is in bijective correspondence with $f'$, and $(P_2,P_1)$ is in bijective correspondence with $P'$. Thus, the statement is immediate if the associated weights are proportional. We have that 
    \begin{align*}
        \kappa(f,P_2,P_1)\propto & \prod_{v\in X^{\brac{0}}}\Bigl[ I_0\bigl(P_1(v) \bigr)+k_1\cdot P_1(v)\cdot I_0 \bigl(f(\epsilon) \bigr)\Bigr] 
        \\&\qquad\cdot
        \prod_{e\in X^{\brac{1}}} \Bigl[ I_0(P_2(e))+k_2\cdot P_2(e)\cdot I_0(\delta f(e))\Bigr].
    \end{align*}

If $e=(g,v)$, then $P_1(v)=P'(e)$, and $I_0(f(e))=I_0(\delta f'(e))$. Similarly, if $e\in E(G)$, $P'(e)=P_2(e)$ and $\delta f'(e)=\delta f(e)$. Dividing each term by $(1-p_1)^{|X^{\brac{0}}|}(1-p_2)^{|X^{\brac{1}}|}$ yields  
    \begin{align*}
        \kappa'\paren{f',P'}\propto 
        &\prod_{e=(g,v)} \Bigl[I_0\bigl(P'(e)\bigr) + \frac{p_1}{1-p_1} P'(e) I_0\bigl(\delta f'(e)\bigr) \Bigr]
        \\ &\qquad \cdot  \prod_{e \in E(G)} \Bigl[I_0 \bigl(P'(e)\bigr) + \frac{p_2}{1-p_2} P'(e) I_0\bigl(\delta f'\paren{e} \bigr) \Bigr].
    \end{align*}
    Since $k_2=\frac{p_2}{1-p_2}$ and $k_1=\frac{p_1}{1-p_1},$ this concludes the proof.
\end{proof}

\section{Properties of the CPP}
We establish a number of basic properties of Coupled Plaquette Percolation.

\subsection{Special Cases of the CPP}
\label{sec:cases}

Let $X$ be a finite cell complex. The $i$-dimensional plaquette random cluster model~\cite{bernoulli-cell-complexes,duncan2025sharp,shklarov} with parameters $p\in\brac{0,1}$ and $q\in \N_{\geq 2}$ is the random $i$-dimensional percolation subcomplex of $X$ whose probability is weighted by the absolute $(i-1)$-homology with coefficients in $\Z_q.$ Specifically, define $\hat{\rho}=\hat{\rho}_{p,q,i,X}$ by 
$$\hat{\rho}\paren{P}\propto   p^{|P|}\paren{1-p}^{|X^{\brac{i}}|-|P|} \abs{H^{i-1}(P;\mathbb{Z}_q)}.$$
The definition of $\hat{\rho}$ is extended to $q=1$ by setting 
 $\hat{\rho}_{p,1,i,X}\propto   p^{|P|}\paren{1-p}^{|X^{\brac{i}}|-|P|}.$
That is, the $q=1$ case of the plaquette random cluster model is Bernoulli plaquette percolation. This naturally occurs as the $q\to 1$ limit of a related model where the choice of cohomology coefficients is in some sense decoupled from the choice of $q$ in exactly the same fashion as the auxiliary model defined in Section~\ref{sec:stochdom} below. 

Let $\pi^2_{p_2,p_1,q,i,X}$ and $\pi^1_{p_2,p_1,q,i,X}$ be the marginal distributions of $P_2$ and $P_1$ when $\paren{P_2,P_1}\sim \rho_{p_2,p_1,q,i,X}.$
By definition, $\pi^2_{p_2,0,q,i,X}$ has the distribution of the plaquette random cluster model, and $\pi^2_{p_2,0,1,i,X}\stackrel{d}{=}\hat{\rho}_{p,i+1,X}.$ Our next result, Proposition~\ref{prop:10}, collects two other special cases.

\begin{prop}\label{prop:10}
If $X$ is a finite cell complex then
    \begin{align*}
        \pi^2_{p_2,0,q,i,X}&\overset{d}{=}\hat{\rho}_{p_2,q,i+1,X},
        \pi^2_{p_2,1,q,i,X} \overset{d}{=}\hat{\rho}_{p_2,1,i+1,X},
    \end{align*}
    and
    \begin{align*}
        \pi^1_{0,p_1,q,i,X} & \overset{d}{=}\hat{\rho}_{p_*\paren{p,q},1,i,X},
    \end{align*}
    where
    $$
        p_*\paren{p,q}= \frac{p}{q-pq+p}.
    $$
    In addition, if $X$ satisfies  $H^{i-1}(X;\,\Z_q)=H^i(X;\,\Z_q)=0$ then
    $$
        \pi^1_{1,p_1,q,i,X} \overset{d}{=}\hat{\rho}_{p_1,q,i,X}.$$
    \end{prop}

\begin{proof}
    Since the presence or absence of $j$-cells in $X$ for $j>i+1$ does not influence the distribution of $\rho$, we may assume that $X=X^{(i+1)}.$ 

    The first claim follows from the observation that  $H^i(P_2,X^{(i-1)})= Z^i(P_2)$ so
    $$\bigl| H^i({P_2,X^{(i-1)}}) \bigr| = \bigl| Z^i(P_2) \bigr| =
    \bigl| H^i(P_2) \bigr| \bigl| B^i(P_2) \bigr| \propto \bigl| H^i(P_2) \bigr|$$
    since $B^i(P_2)=B^i\paren{X}$ does not depend on $P_2.$ To obtain the second claim, note that the only cochain compatible with $\paren{P_2,X^{\paren{i}}}$ is identically zero so $H^i({P_2,X^{(i)}})=0.$

    For the third claim, note that when $P_2$ contains no $(i+1)$-cells, the only constraint on compatible $i$-cochains is that they vanish on the $i$-cells of $P_1.$ Thus
    $$
        \bigl| H^i({X^{(i)},P_1})\bigr| = \bigl|C^i ( X,P_1)\bigr|=q^{|{X^{\brac{i}}|}-|{P_1}|},
        $$
    and hence
    \begin{align*}
        \rho_{0,p_1,q,i,X}(X^{(i)},P_1)&\propto \paren{\frac{p_1}{1-p_1}}^{|P_1|}q^{|X^{(i)}|-|P_1|}\propto   \paren{\frac{p_*\paren{p,q}}{1-p_*\paren{p,q}}}^{|P_1|}
        =\rho_{p_*\paren{p,q},0,1,i-1,X}.
    \end{align*}
    
To prove the fourth claim, we will demonstrate that $H^{i-1}(P_1)\cong H^{i}(X,P_1).$ Readers familiar with algebraic topology will recognize that the following argument is an application of the long exact sequence of a pair.

Let $f \in Z^{i-1}\paren{P_1}$ and consider the coboundary $\delta f.$ $\delta f$ vanishes on each $i$-cell of $P_1$ by definition and $\delta\paren{\delta f}=0$ so $\delta f\in Z^i\paren{X,P_1}=H^i\paren{X,P_1}.$ Since $\delta\circ\delta=0,$ $\delta$ sends each element of the homology class of $f$ to the same homology class in $H^i\paren{X,P_1}.$ That is, the coboundary map induces a well-defined map
$$\delta^*:H^{i-1}\paren{P_1}\to H^i\paren{X,P_1}.$$
We will show that $\delta^*$ is an isomorphism under the assumption that $H^{i-1}(X)=H^i(X)=0.$ First, if $\delta f=0$ as an element of $H^i\paren{X,P_1}=Z^i\paren{X,P_1}$ then $\delta f$ is the zero cochain. It follows that $f\in Z^{i-1}\paren{X}.$ Since $H^{i-1}\paren{X}=0,$ we must have that $f\in B^{i-1}\paren{X}.$ As $B^{i-1}(X)=B^{i-1}(P_1)$ $\brac{f}=0\in H^{i-1}\paren{P_1}$ and so $\delta^*$ is injective. On the other hand, let $g\in  H^{i}(X,P_1)=Z^i\paren{X,P_1}\subset Z^i\paren{X}.$ Since $H^i\paren{X}=0,$ there is a $f\in C^{i-1}\paren{X}$ satisfying $\delta f=g.$ By construction, $\delta f$ vanishes on each $i$-cell of $P_1$ so $f\in Z^1\paren{P_1}$ and $\delta^* \brac{f}=g.$ It follows that $\delta^*$ is surjective and and thus an isomorphism.

\end{proof}

\subsection{Positive Association}\label{sec:positive}
In this section we prove Theorem~\ref{thm:positive}, which states that the CPP is positively associated. We begin with a topological lemma.

\begin{lemma} \label{lem:homlattice}
Let $Z$ be a finite cell complex. If $X,Y$ are $(i+1)$-dimensional percolation subcomplexes of $Z$ and $A,B$ are $i$-dimensional percolation subcomplexes of $Z$ then
$$|H^i(X\cup Y,A\cup B)||H^i(X\cap Y,A\cap B)|\geq  |H^i(X,A)||H^i(Y,B)|$$
\end{lemma}

\begin{proof}
As we are assuming that $q$ is prime, that $H^i(S,T;\,\Z_q)$ is a $\Z_q$-vector space, and that $|H^i(S,T;\,\Z_q)|=q^{b_i(S,T;\,\Z_q))}$, it suffices to show that
$$b_i(X\cup Y,A\cup B)+b_i(X\cap Y,A\cap B)\geq b_i(X,A)+b_i(Y,B).$$

We prove this by means of the following Mayer--Vietoris sequence on relative cohomology (see Section~\ref{sec:MV}).
$$H^{i+1}(X\cup Y,A\cup B) \xleftarrow{\delta} H^{i}(X\cap Y,A\cap B) \xleftarrow{\chi}H^{i}(X,A)\oplus H^{i}(Y,B) \xleftarrow{\phi}H^{i}(X\cup Y,A\cup B)$$

By exactness and the first isomorphism theorem we obtain the three equations
\begin{align*}
    \begin{cases}
        b_i(X\cup Y, A\cup B)=\rank \phi+\nlty\phi \cr
        b_i(X,A)+b_i(Y,B)=\rank \chi+\nlty\chi=\rank\chi+\rank\phi\cr
        b_i(X\cap Y,A\cap B)=\rank\delta+\nlty\delta=\rank\delta+\rank\chi.
\end{cases}
\end{align*}
Thus 
\begin{align*}
b_i(X\cup Y,A\cup B)+b_i(X\cap Y,A\cap B)&=\rank\phi+\nlty\phi+\rank\delta+\rank\chi\\
&\geq \rank\chi+\rank\phi\\
&=\rank\chi +\nlty \chi\\
&=\dim H^{i}(X,A)\oplus H^{i}(Y,B)\\
&=b_i(X,A)+b_i(Y,B).
\end{align*}

\end{proof}

\begin{proof}[\textbf{Proof of Theorem \ref{thm:positive}}]

Let $Z$ be a finite cell complex, let $q$ be a prime integer, let $i$ be a non-negative integer, and let $p_2,p_1\in \brac{0,1}.$ Set $\rho=\rho_{p_2,p_1,q,i,d,Z}.$ 
To show that $\rho$ is positively associated, it suffices to show that the lattice condition holds~\cite{FKG71}. To this end, let $X,Y$ be $(i+1)$-dimensional percolation subcomplexes of $Z$ and let $A,B$ are $i$-dimensional percolation subcomplexes of $Z.$ 
Let
$$r^{X,Y}_{A,B}=\frac{\rho(X\cup Y,A\cup B)\rho(X\cap Y,A\cap B)}{\rho(X,A)\rho(Y,B)}.$$
We again work in $k_2,k_1$ coordinates for brevity and clarity. Using this notation and applying Lemma~\ref{lem:homlattice}, we have
\begin{align*}
r^{X,Y}_{A,B}&=\frac{k_1^{|A\cup B|}k_2^{|X\cup Y|}|H^i(X\cup Y,A\cup B)|k_1^{|A\cap B|}k_2^{|X\cap Y|}|H^i(X\cap Y,A\cap B)|}{ k_1^{|A|}k_1^{|X|}|H^i(X,A)|k_1^{|B|}k_2^{|Y|}|H^i(Y,B)|}\\
&=\frac{k_1^{\abs{A\cup B} +\abs{ A\cap B}}}{k_1^{\abs{A}+\abs{B}}}\frac{k_2^{\abs{X\cup Y} +\abs{ X\cap Y}}}{k_2^{\abs{X}+\abs{Y}}}\frac{|H^i(X\cup Y,A\cup B)||H^i(X\cap Y,A\cap B)|}{ |H^i(X,A)||H^i(Y,B)|}\\
&=\frac{|H^i(X\cup Y,A\cup B)||H^i(X\cap Y,A\cap B)|}{ |H^i(X,A)||H^i(Y,B)|}
\geq 1
\end{align*}
by Lemma~\ref{lem:homlattice}. This concludes the proof.
\end{proof}

\subsection{Stochastic Domination}\label{sec:stochdom}

Theorem~\ref{thm:stochdom} is a special case of a more general statement for an auxiliary model where the parameter $q$ is allowed to vary independently of the coefficient field for cohomology.

\begin{definition}[Auxiliary Model]
Let $X$ be a finite cell complex, let $q$ be a prime integer, let $i$ be a non-negative integer, let $0\leq p_2,p_1\leq 1$ and $r$ be a non-negative real number. 
Define $\hat{\rho}=\hat{\rho}_{p_2,p_1,r,i,q,X}$ by
$$\hat{\rho}({P_2,P_1})\propto  p_1^{|P_1|}(1-p_1)^{|X^{(i)}|-|P_1|} p_2^{|P_2|}(1-p_2)^{|X^{(i+1)|}-|P_2|}r^{b_i(P_2,P_1;\mathbb{Z}_q)}.$$
\end{definition}

Note that $\hat{\rho}_{p_2,p_1,q,i,q,X}=\rho_{p_2,p_1,q,i,X}$, and when $r=1$, $\hat{\rho}$ is independent percolation on $i$ and $(i+1)$-cells with probabilities $p_2,p_1$ respectively. The proof of positive association for $\rho$ extends immediately to~$\hat{\rho}$ when $r\geq 1.$ 

For the subsequent proof, we require Holley's theorem, as stated in~\cite{G06}. Let $Z$ be a finite set. For a probability measure $\pi$ on $\Omega\coloneqq \set{0,1}^{\mathrlap{Z}},\,$ $z\in Z,$ and $\xi\in\Omega$ define the one-point conditional probability by
$$
    \pi^{\xi}\paren{z}
    \coloneqq
    \pi\bigl(\omega\paren{z}=1\mid \omega|_{Z\smallsetminus z}=\xi|_{Z\smallsetminus z} \bigr).
$$ 

\begin{theorem}[Holley's theorem] Let $Z$ be a finite set, and let $\pi_1,\pi_2$ be strictly positive measures on $\Omega=\{0,1\}^Z$. Then $\pi_1 \leq_{st}\pi_2$ if and only if 
$$\pi_1^{\xi}(z)\leq  \pi_2^{\zeta}(z)$$
for all $z\in Z$ and all pairs $\xi, \zeta\in \omega$ with $\xi\leq\zeta.$ 
\end{theorem}

\begin{prop}\label{prop:stochdom-support}
$\hat{\rho}$ is stochastically decreasing in $r$ for fixed $p_2,p_1$. If $r$ is allowed to vary while $\frac{p_1}{r(1-p_1)+p_1}$ and $\frac{p_2}{r(1-p_2)+p_2}$ are fixed then $\hat{\rho}$ is stochastically increasing in $r$. When $r$ is fixed, $\hat{\rho}$ is stochastically increasing in $p_2$ and $p_1$. 
\end{prop}
\begin{proof}
If $\pi_1,\pi_2$ are both distributed as $\hat{\rho}$ for some choice of parameters, $z\in X^{\brac{i+1}}\cup X^{\brac{i}},$ and ${\zeta,\xi \in \{0,1\}^{X^{\brac{i+1}}}\times \{0\}^{X^{\brac{i+1}}}}$ with $\xi\leq \zeta$ then  
$$\pi_2^{\xi}(z)\leq \pi_2^\zeta(z)$$
by positive association. Consequently, the hypothesis for Holley's theorem simplifies to the claim that 
$$\pi_1^{\xi}(z)\leq \pi_2^{\xi}(z)$$
for all $z$ and $\xi.$ 

To establish this, there are two cases, depending on whether $z$ is an $i$-cell or $(i+1)$-cell. For clarity, we denote the one point conditional for an $i$-cell $\epsilon$ by $\hat{\rho}_i^{\xi}\paren{\epsilon}$ and the one point conditional for an $(i+1)$-cell $\sigma$ by $\hat{\rho}_{i+1}^{\xi}\paren{\sigma}.$ 

We can calculate these probabilities explicitly from the definition. To this end, let $S$ and $T$ be $i$-dimensional and $(i+1)$-dimensional percolation subcomplexes, let $\xi$ be the corresponding element of  $\{0,1\}^{X^{\brac{i+1}}}\times \{0,1\}^{X^{\brac{i}}},$ let $\epsilon$ be an $i$-cell, and let $\sigma$ be an $(i+1)$-cell. Then
\begin{align*}
\hat{\rho}_i^{\xi}\paren{\epsilon}&=\frac{1}{1+\frac{\hat{\rho}(S\smallsetminus \epsilon,T)}{\hat{\rho}(S\cup \epsilon,T)}} \quad \text{and} \quad
\hat{\rho}_{i+1}^{\xi}\paren{\sigma}=\frac{1}{1+\frac{\hat{\rho}(S,T\smallsetminus \sigma)}{\hat{\rho}(S,T\cup \sigma)}}.
\end{align*}
Since
\begin{align*}
\frac{\hat{\rho}(S\cup\epsilon,T)}{\hat{\rho}(S\smallsetminus \epsilon,T)}&=\frac{p_1}{1-p_1}r^{b_i(T,S\cup\epsilon)-b_i(T,S\smallsetminus \epsilon)}
\end{align*}
and
\begin{align*}
    \frac{\hat{\rho}(S,T\cup\sigma)}{\hat{\rho}(S,T\smallsetminus \sigma)}&=\frac{p_2}{1-p_2}r^{b_i(T\cup\sigma ,S)-b_i(T\cup \sigma,S)},
\end{align*}
the one-point conditionals are determined by $b_i(T,S\cup\epsilon;\,\Z_q)-b_i(T,S\smallsetminus \epsilon;\,\Z_q)$ and $b_i(T\cup\sigma,S;\,\Z_q))-b_i(T\smallsetminus\sigma,S;\,\Z_q)$.

Next, let $P$ and $Q$ be $(i+1)$- and $i$-dimensional percolation subcomplexes of $X,$ respectively. An element of $Z^i(P,Q;\,\Z_q)=H^i(P,Q;\,\Z_q)$ is an $i$-cochain ${f\in C^i(X;\,\Z_q)}$ so that $\delta f(\sigma)=0$ for all $\sigma\in P^{(i+1)}$ and $f\paren{\epsilon}=0$ for all $\epsilon \in Q.$ That is, the collection of cochains compatible with $\paren{P,Q}$ is a linear subspace of $C^i(X;\,\Z_q)$ determined by these equations. Adding a single $i$-cell to $Q$ or a single $(i+1)$-cell to $P$ adds a single linear equation and can thus either leave the dimension $b_i(P,Q;\,\Z_q)$ unchanged or decrease it by one. 
Therefore, $b_i(T,S\cup\epsilon;\,\Z_q)-b_i(T,S;\,\Z_q)$ and $b_i(T\cup\sigma,S;\,\Z_q)-b_i(T,S;\,\Z_q)$ are each either equal to $0$ or $-1$. 

From this, it follows that
$$
\hat{\rho}_i^{\xi}\paren{\epsilon}=\begin{cases}\frac{p_1}{r(1-p_1)+p_1}& \text{if } b_i(T,S\cup\epsilon;\,\Z_q)=b_i(T,S;\,\Z_q)-1\\
p_1 & \text{if } b_i(T,S\cup\epsilon;\,\Z_q)=b_i(T,S;\,\Z_q)
\end{cases}
$$
and
$$\hat{\rho}_{i+1}^{\xi}\paren{\sigma}=\begin{cases}\frac{p_2}{r(1-p_2)+p_2}& \text{if } b_i(T\cup\sigma ,S;\,\Z_q)=b_i(T,S;\,\Z_q)-1\\
p_2 & \text{if } b_i(T\cup\sigma ,S;\,\Z_q)=b_i(T,S;\,\Z_q).
\end{cases}$$
Both \( \hat{\rho}_i^{\xi}\paren{\epsilon}\) and \( \hat{\rho}_{i+1}^{\xi}\paren{\sigma} \) are monotone increasing in $p_2,p_1$ respectively, so if $p_2'\leq p_2$ and $p_1'\leq p_1$ then $\rho_{p_2',p_1'}\leq_{st}\rho_{p_2,p_1}$. Similarly,  they are decreasing in $r$, so if $r>r'$, then $\rho_{p_2,p_1,r}\leq_{st}\rho_{p_2,p_1,r'}$.  If we fix $c_1(p_1,r)=\frac{p_1}{r(1-p_1)+p_1}$ and $c_2(p_2,r)=\frac{p_2}{r(1-p_2)+p_2}$ while allowing $r$ to vary, $p_2$ and $p_1$ are both increasing in $r$, so if $r>r'$, then $\rho_{p_2,p_1,r}\geq_{st}\rho_{p_2,p_1,r'}$.

\end{proof}

\begin{proof}[\textbf{Proof of Theorem \ref{thm:stochdom}}]
The desired conclusion follows immediately from Proposition~\ref{prop:stochdom-support}, as $\rho_{p_2,p_1,q,i,X}\stackrel{d}{=}\hat{\rho}_{p_2,p_1,q,i,q,X}$ and $\hat{\rho}_{p_2,p_1,q,i,1}$ is distributed as a pair of independent Bernoulli plaquette percolations.

\end{proof}

\subsection{Duality}\label{sec:duality}

For the proof of Theorem~\ref{thm:duality}, we will need the following results from algebraic topology.  We recall some notation from the introduction. $\T=\T_N^d$ is the discrete torus of width $N.$ For a $j$-dimensional percolation subcomplex $P$ of $\T,$ $P^{\bullet}$ is the $(d-j)$-dimensional dual complex which contains a dual $l$-cell for each $\paren{d-l}$-cell not contained in $P.$


\begin{prop} \label{lem:bullet}
 Let \( P_1 \) and \( P_2 \) be percolation subcomplexes of $\T$ of dimension $i$ and $(i+1)$ respectively. Then
$$H^{j}(P_1^\bullet,P_2^\bullet)\cong H_{d-j}(P_2,P_1).$$
\end{prop}

\begin{proof}
By Lemma 7 of~\cite{duncan2025homological}, the complements $\T\smallsetminus P_2$ and $\T\smallsetminus P_1$ deformation retract to $P_2^\bullet$ and $P_1^{\bullet}$ respectively. From this, it follows that 
$$H^{j}(P_1^\bullet, P_2^{\bullet})\cong H^{j}(\T\smallsetminus P_1,\T\smallsetminus P_2).$$
From relative Alexander duality (Equation~\ref{eqn:alexander}), we have 
$$H^{j}(\T\smallsetminus P_1,\T\smallsetminus P_2)\cong H_{d-j}(P_2,P_1).$$
Combining the above observations, the desired conclusion follows.
\end{proof}

\begin{lemma} \label{lem:ranks}
    Let \( P_1 \) and \( P_2 \) be percolation subcomplexes of $\T$ of dimension $i$ and $(i+1)$ respectively. Then
    $$
        b_{i+1}(P_2,P_1)=b_{i}(P_2,P_1)+|P_2|+|P_1|-|\T^{(i+1)}|.
    $$
\end{lemma}

\begin{proof}
    Note first that since $P_2$ and $P_1$ coincide on the $(i-1)$-skeleton of $\T$ , we have $H_j(P_2,P_1)\cong 0$ for all $j\leq i-1.$ By Proposition~\ref{prop:euler}, we have
    $$\chi(Y)=\sum^d_{j=0}(-1)^j b_j(Y).$$
    Since  $P_2$ and $P_1$ form a good pair, it follows from~\cite[Proposition 2.22]{H02} that $${H_j(P_2,P_1)\cong \tilde{H}_j(P_2/P_1)},$$ where $X/A$ denotes the quotient space and the reduced homology $\tilde{H}_j(P_2/P_1)$ agrees with $H_j\paren{P_2/P_1}$ except for $j=0$ when $H_j\paren{P_2/P_1;\,G}\cong \tilde{H}_j(P_2/P_1)\oplus G.$
    Combining these facts, we obtain
    $$
        \chi(P_2/P_1)=1+\sum^d_{j=0}(-1)^jb_j(P_2,P_1)
        =
        (-1)^i\bigl(b_i(P_2,P_1)-b_{i+1}(P_2,P_1)\bigr)+1.
    $$
    However, the Euler characteristic of the quotient also satisfies the identity
    $$
        \chi\paren{P_2/P_1} = \chi(P_2)-\chi(P_1)+1
    $$
    since a cell complex structure for $P_2/P_1$ can be found by replacing all cells of $P_1$ with a single vertex. 
It thus follows that
    $$ 
        \chi(P_2/P_1)=1+\sum_{j=0}^d (-1)^j \bigl( P_2^{(j)}-P_1^{(j)} \bigr)
        =
        (-1)^i\bigl(|\T^{(i)}|-|P_2|-|P_1|\bigr)+1,
    $$
    where the numbers of cells in dimensions less than $i$ cancel. Hence 
    $$
        (-1)^i(b_i(P_2,P_1)-b_{i+1}(P_2,P_1))+1 = (-1)^i \bigl(|\T^{(i)}|-|P_2|-|P_1| \bigr)+1.
    $$
    This concludes the proof.
\end{proof}

\begin{proof}[\textbf{Proof of Theorem~\ref{thm:duality}}]
We will use the parameters $k_2,k_1$ rather than $p_2,p_1$ to simplify notation. To relate $\rho(P_2,P_1)$ and $\rho^\bullet(P_1^\bullet,P_2^\bullet)$, we will find a formula comparing the sizes of the relative cohomology groups $H^i(P_2,P_1;\,\Z_q)$ and $H^{d-i-1}(P_1^\bullet,P_2^\bullet;\,\Z_q)$.

    Recall that we are assuming that $q$ is prime, and thus $\Z_q$ is a field and 
    $$
        \abs{H_j(X,A;\,\Z_q)}= \abs{H^j(X,A;\,\Z_q)}=q^{b_j(X,A;\,\Z_q)}
    $$
    for all $j.$ Thus, by Proposition \ref{lem:bullet} and Lemma \ref{lem:ranks}, when ${d-i-1>0},$ we have
    \begin{align*}
        &|H^{d-i-1}(A^\bullet, X^\bullet)|
        =
        |\tilde H^{d-i-1}(A^\bullet, X^\bullet)|
        =
        |H_{i+1}(X,A)|
        =
        |H^{i+1}(X,A)|\\&\qquad=|H^{i}(X,A)|q^{|X|+|A|-|\T|^{(i+1)}+(-1)^i}.
    \end{align*}

Now we seek to find a choice of $k_2',k_1'$ so that $$\rho_{k_2,k_1,d}(P_2,P_1)=\rho _{k_2',k_1',d-i+1}(P_1^\bullet,P_2^\bullet).$$ Towards that end, we compute
    \begin{align*}
        & k_1^{|P_1|}k_2^{|P_2|}|H^i(P_2,P_1)|=k_1'^{|P_2^\bullet|}k_2'^{|P_1^\bullet|}|H^{d-i-1}(P_1^\bullet,P_2^\bullet)|\\
        &\qquad= k_1'^{|\T^{(i+1)}|-|P_2|}k_2'^{|\T^{(i)}|-|P_1|}|H^i(P_2,P_1)|q^{|P_2|+|P_1|-|\T^{(i+1)}|} \\
        &\qquad\propto k_1'^{-|P_2|}k_2'^{-|P_1|}|H^i(P_2,P_1)|q^{|P_2|+|P_1|}.
    \end{align*}
    To maintain proportionality, we have $k_1^{|P_1|}=k_2'^{-|P_1|} q^{|P_1|}$, so $k_2'=q/k_1$, and symmetrically $k_1'=q/k_2$ gives us an exact duality between $\rho_{k_2,k_1}(P_2,P_1)$ and $\rho_{k_2',k_1'}(P_1^\bullet,P_2^\bullet)$. If $d-i-1=0$ we get an extra constant factor of $q$, which doesn't affect the proportionality. This concludes the proof.
\end{proof}

\section{Applications}\label{sec:applications}
The coupling with the CPP leads to relatively simple, geometrically intuitive proofs of several standard results for the Potts lattice Higgs model, including the \( \mathbb{Z}_2 \) lattice Higgs model. Throughout this section, let $X$ be a finite cell complex, $q$ be a prime integer, $\beta_2,\beta_1\geq 0,$ and $\mu=\mu_{\beta_2,\beta_1,q,i,X}.$ 

\begin{prop}[Griffith's Second Inequality]
    Let $\gamma_1,\gamma_2\in C_i(X;\,\Z_q).$ Then
    $$
        \E_\mu(W_{\gamma_1} W_{\gamma_2})
        \geq
        \E_\mu(W_{\gamma_1}) \E_\mu(W_{\gamma_2}).
    $$
\end{prop}
\begin{proof}
    By Theorem \ref{thm:W=V} and Theorem \ref{thm:positive},  we have 
    \begin{equation*}
        \E_\mu(W_{\gamma_1}) \E_\mu(W_{\gamma_2}) = \rho(V_{\gamma_1})\rho(V_{\gamma_2})\leq \rho(V_{\gamma_1}\cap V_{\gamma_2}).
    \end{equation*}
    Next, since $V_{\gamma_1}\cap V_{\gamma_2}\subset V_{\gamma_1+\gamma_2}$ and $W_{\gamma_1}W_{\gamma_2}=W_{\gamma_1+\gamma_2}$ it follows that 
    $$ \rho(V_{\gamma_1}\cap V_{\gamma_2})\leq  \rho(V_{\gamma_1+\gamma_2})=\E_\mu(W_{\gamma_1} W_{\gamma_2}),
    $$ 
    which implies the desired conclusion.
\end{proof}

\begin{prop}\label{prop: perimeter}
    Let $X$ be a subcomplex of $\Z^d$ or $\T_N^d$ and let $\gamma \in C_i\paren{X}.$ Denote the number of $i$-cells in the support of $\gamma$ by $|\gamma|.$ For any $\beta_2,\beta_1>0,$ there exist $c_1,c_2>0$ depending on $\beta_2,\beta_1,i$ and $d$ but not on $q, N,$ or $X$ so that
    $$
        e^{-c_1 |\gamma|}
        \leq \E_\mu (W_{\gamma})
        \leq e^{-c_2 |\gamma|}.
    $$
    \end{prop}
\begin{proof}
    Suppose that $\gamma$ is supported on $e_1,\ldots,e_n$ and let Let $A_j$ be the event that $e_j\in P_1.$ Then $\cap_{j=1}^N A_j \subset V_{\gamma}$ so
    $$
        \E_\mu(W_{\gamma})
        =
        \rho\paren{V_{\gamma}}\geq \rho\paren{\cap_{i=1}^n A_i} > 
        \paren{\frac{p_1}{q}}^n=e^{-\log\paren{q/p_1}|\gamma|},
    $$
    where the final inequality follows by positive association.

On the other hand, for $V_{\gamma}$ to occur each $i$-cell $e_j$ must be either in $P_1$, or adjacent to an $(i+1)$-cell in $P_2$. This is less probable than the same requirement being satisfied on a subset of $\gamma$. The adjacency graph on $i$-cells which are connected by an edge if they are incident to an $(i+1)$-cell is $2(i+1)$-colorable, by alternating the colors assigned to parallel cells on adjacent $(i+1)$-cells, so we can always choose a set $\gamma'$ of size at least $\frac{N}{2(i+1)}$ so that no two $i$-cells in $\gamma'$ are adjacent to the same $(i+1)$-cell. Next, denote by $B_j$ the event that at least one of the $2(d-i)$ $(i+1)$-cells incident to $e_j$ is contained in $P_2$ and let
    $$
        c_2 =
        -\frac{\log({1-\paren{1-p_2}^{2(d-i)}\paren{1-p_1}}}{2(i+1)}.
    $$
    Then
    $$
        \rho\paren{V_\gamma} 
        \leq 
        \rho\paren{\cap_{i=1}^n B_i\cup A_i}
        \leq 
        \prod_{i=1}^{n/(2i+2)} \Bigl({1-\paren{1-p_2}^{2(d-i)}\paren{1-p_1}} \Bigr)
        =
        e^{-c_2 |\gamma|} ,
    $$
    where we used Theorem~\ref{thm:stochdom} to compare \( \rho \) to independent plaquette percolation. 
\end{proof}

\begin{prop} Let $\gamma\in C_i(X;\,\Z_q).$ Then $\E_{\mu_{\beta_2,\beta_1,q,i,X}}(W_\gamma)$ is increasing in ${\beta_2}$ and ${\beta_1}$.
\end{prop}

\begin{proof}
    By Theorem~\ref{thm:stochdom}, $\rho_{p_2,p_1}$ is stochastically increasing in $p_2$ and $p_1$. The event $V_\gamma$ is increasing, so $\rho_{p_2,p_1}(V_\gamma)$ is increasing in $p_1$ and $p_2$. By Theorem \ref{thm:W=V} and the fact that $p_2$ and $p_1$ are increasing in ${\beta_2}$ and ${\beta_1}$ respectively,  it follows tat $\E_{\mu({\beta_2},{\beta_1})}(W_\gamma)$ is increasing in ${\beta_2}$ and ${\beta_1}$.
\end{proof}

Before showing that the Marcu-Fredenhagen Ratio exhibits a phase transition for $i=1,$ we prove that its natural analogue for $i\geq 2$ has trivial behavior. 

\begin{prop}\label{prop: MF for i geq 2}
Let $i\geq 2$ be fixed, let $n\in 2\N,$ let $q_n=\partial \bigl(\brac{0,n}^{i+1}\times\set{0}^{d-i-1})$ and let $q_n'$ and $\gamma_n''$ be the chains formed by the upper and lower halves of $q_n$ (see Figure~\ref{figure: Wilson lines MF}).
$$
    R\paren{\beta_2,\beta_1,n}
    =
    \frac{\mathbb{E}_{\mu}(W_{q_n'})\mathbb{E}_{\mu}(W_{q_n''})}{\mathbb{E}_{\mu}(W_{q_n})}=\frac{\mathbb{E}_{\mu}(W_{q_n'})^2}{\mathbb{E}_{\mu}(W_{q_n})}.$$
Then 
    \[
    \lim_{n \to \infty } \hat R(p_2,p_1,n) =0.
    \]
\end{prop} 

\begin{proof}
    By Theorem~\ref{thm:positive}, we have
	\begin{align*}
		& R(p_2,p_1,n) = \frac{\rho(  V_{q_n'}) \rho( V_{q_n''})}{\rho (V_{q_n'+q_n''})}
        =
        \frac{\rho^2(P \in V_{q_n'},\, Q \in V_{q_n''})}{\rho^2(P \in V_{q_n'+q_n''})}
        \leq
        \frac{\rho(P \in V_{q_n'} \cap  V_{q_n''})}{\rho(P \in V_{q_n'+q_n''})}
        \\&\qquad =
        \rho(P \in V_{q_n'} \mid P \in V_{q_n'+q_n''}).
	\end{align*}
    Let \( \gamma_n = \partial q_n' \) the boundary between \( q_n' \) and \( q_n''. \) On the event \( P \in V_{q_n'} \cap  V_{q_n''},\) for each $(i-1)$-cell \( \sigma \in \support \gamma_n, \) there must exist at least one i-cell \(\tau \in  P \) incident to $\sigma.$  We now sample \( P \sim \rho(\cdot \mid P  \in V_{q_n})\) as follows. First, let \( P_0 \sim \rho(\cdot \mid P_0 \in V_{q_n})\). For \( \sigma \in \support\gamma_n\) let  $T_{\sigma,i+1}$ be the collection of $(i+1)$-cells in the star of $\sigma$ (that is, the collection of $(i+1)$-cells containing $\sigma$ as a face of co-dimension two) and let $T_{\sigma,i}$ be the collection of $i$-cells incident to an $(i+1)$-cell in $T_{\sigma,i+1}.$ 
    We~abuse notation and denote by $\partial T_{\sigma,i+1}\subset T_{\sigma,i}$ the $i$-cells incident to exactly one $(i+1)$-cell of~$T_{\sigma,i+1}.$ 
    
    Let \( E_n \subseteq \support \gamma_n\) be a collection of $i$-cells such that for any \( \{ \sigma,\sigma' \} \subseteq E_n,\) the sets \( T_{\sigma,j}\) and  \( T_{\sigma,j}\) are disjoint for all \( j \in \{ i,i+1 \}.\) Now obtain \(P\) from \( P_0\) by resampling from \( \rho(\cdot \mid V_{q_n'+q_n''} )\) on the sets \( T_{\sigma,{i}}\) and  \( T_{\sigma,{i+1}}\) for all \( \sigma \in E_n.\) Then, for each \( \sigma \in E_n,\) with strictly positive probability \( p \) (not depending on \( n \)), we have \( T_{\sigma,i+1} \subseteq P_2\) and \( P_1 \cap T_{\sigma,i} = \partial R_{\sigma,i+1} \) However, on this event, we cannot have \( P \in V_{q_n'}.\) Hence
    \[
     R(p_2,p_1,n) \leq (1-p)^{|E|}.
    \]
    Now note that since \( i \geq 2, \) we have \( \lim_{n \to \infty} |\gamma_n| = \infty\), and we can thus choose \( E_n\) such that \( \lim_{n \to \infty}|E| = \infty.\) This concludes the proof.
\end{proof}

\section{Phase Transition of the Marcu--Fredenhagen Ratio}
\label{sec:ratio}
For an $i$-chain $\gamma$ denote by $|\gamma|$ the number of $i$-cells in the support of $\gamma.$ This is sometimes called the perimeter of $\gamma.$ Recall that we let $n\in 2\N,$ $\gamma_n=\partial (\brac{0,n}^2\times \set{0}^{d-2}$, and let $\gamma_n'$ and $\gamma_n''$ be the paths formed by the upper and lower halves of $\gamma_n$ (see Figure~\ref{figure: Wilson lines MF}). 
Recall also the definitions of the Marcu--Fredenhagen ratio \( R(\beta_2,\beta_1,n) \) and the topological Marcu--Fredenhagen ratio \( \hat R(p_2,p_1,n)\) from Definition~\ref{def: MF} and Definition~\ref{def: MF top} respectively, and note that by Theorem~\ref{thm:W=V}, they are equal.

Fix $i=1,$ $d$, and $q,$ let $\Lambda_N=[-N,N]^d,$ and let $\rho=\rho_{p_2,p_1}=\rho_{p_2,p_1,q}=\rho_{p_2,p_1,q,1,\Z^d}$ be the weak limit of the measures  $\rho_{p_2,p_1,q,1,\Lambda_N^d} ,$ which exists by a standard stochastic domination argument. Recall that the area of a cycle $\gamma\in B_i(A;\,\Z_q)$ is the minimum number of plaquettes in the support of a $(i+1)$-chain $\tau$ so that $\partial \tau=\gamma.$  An central result which will be useful in this section is the following theorem, which states that pure Potts lattice gauge theory has a phase transition between a region with perimeter law and area law.
We note that there was a typo in the original statement of the theorem.

\begin{theorem}[Theorem~7 of~\cite{duncan2025topological}]
    \label{thm:areaperimeter}
    Let $q$ be a prime integer and $i<d.$ Consider the plaquette random cluster model $\hat{\rho}\coloneqq \hat{\rho}_{p,q,i,\Z^d}=\rho_{p_2,0,q,i,\Z^d}.$  There exist positive, finite constants $c_1=c_1(p_2,q,i,d),c_2=c_2(p_2,q,i,d)$ and $0<p'=p'\paren{q,i,d}\leq p''=p''\paren{q,i,d}<1$ so that, for hyperrectangular $(i-1)$-boundaries $\gamma$ in $\Z^d,$
    \begin{equation}
    \label{eq:wilsoninequalities1}
        \exp(-c_1 \mathrm{Area}(\gamma)) 
        \leq 
        \hat{\rho}(V_\gamma) 
        \leq  
        \exp(-c_2 |\gamma|)
    \end{equation}
    for all $p \in \paren{0,1}$, and such that
    \begin{align*}
        \begin{cases}
            -\frac{\log\paren{ \hat{\rho}(V_\gamma)}}{|\gamma|)}\;\,=\;\, \Theta\paren{1} & \text{if } p > p''
            \cr
            -\frac{\log\paren{ \rho(V_\gamma)}}{\mathrm{Area}(\gamma)}\;\, \rightarrow \;\, c_1 & \text{if } p<p'.
        \end{cases}
    \end{align*}
\end{theorem}

\begin{prop}\label{prop:c3}
    Assume the same notation as in the previous theorem. If $i=1$ then
    $$
        \lim_{p\to 0} c_1(p,q,1,d)=\infty.
    $$
\end{prop}

\begin{proof}
    By Theorem~\ref{thm:stochdom}, it suffices to consider the case $q=1.$ For two edges $e_1,e_2$, write $e_1\leftrightarrow e_2$ if they are connected by a path of open plaquettes each meeting at an edge. By standard arguments, when $p_2$ is sufficiently small, then $\rho \paren{e_1\leftrightarrow e_2}\leq e^{-c(p) d\paren{e_1,e_2}}$, where $d\paren{e_1,e_2}$ is the distance between the two closest pairs of incident vertices of the two edges and $\lim_{p\to 0}c(p)=\infty.$ 

    We temporarily reset $\gamma_n$ to be $\partial \brac{0,n}^2\times\set{0}^{d-2}.$ Let $e_j$ be the edge between $\paren{j-1,0,0,\ldots,0}$ and   $\paren{j,0,0,\ldots,0}$ for $j=1,\ldots,n$ and let $e_j'$ be the edge between $\paren{j-1,n,0,\ldots,0}$ and   $\paren{j,n,0,\ldots,0}.$ It suffices to show that if $V_{\gamma_n}$ occurs then the events $e_j\leftrightarrow e_j'$ occur with disjoint witnesses. From that claim, it follows from the BK inequality that
    $$
        \rho(V_{\gamma_n})
        \leq \rho\bigl(\square_{j=1}^{n} e_j\leftrightarrow e_j' \bigr)
        \leq (e^{-c(p) n})^n
        =
        e^{-c(p)\mathrm{Area}(\gamma)},
    $$
    where we denote the event that the events $A_1,\ldots, A_n$ occur with disjoint witnesses by $\square_{j=1}^n A_j.$
    
    We now prove the claim. To this end, let $S_j = \paren{V,E}$ be the graph where the vertices $V$ are the edges of $\Z^d$ that intersect the hyperplane $x=j-1/2,$ and the edges $E$ are the plaquettes that intersect that hyperplane. We will show that if $V_{\gamma_n} $ occurs, then so does $ e_j \xleftrightarrow{S_j} e_j'.$ To see why, suppose that $V_{\gamma_n}$ occurs and is witnessed by a $2$-chain $\tau=\sum_{a_i} \sigma_j$ where $a_i\in \{1,\ldots,q-1 \}.$ Let $\sigma_1',\ldots,\sigma_k'$ be the plaquettes of $S_j$ in the support of $\tau,$ and denote their coefficients by $a_1',\ldots a_n'.$ Each plaquette $\sigma_k'$ is incident to $2a_k'$ edges of $S_j$, half with orientation $1$ and half with orientation $-1.$ 
    On the other hand, since $\partial \tau=\gamma_n$, every edge $e\in V\setminus \{e_j,e_j' \}$ is incident to $0 \mod q$ plaquettes when counted with multiplicity and orientation, whereas $e_j$ is incident to $\pm 1$ and $e_j'$ is incident to $\mp 1$ plaquettes. This defines a mod $q$ flow on $S_j$ with source and sink $e_j$ and $e_j'$ respectively. It follows from standard arguments that $e_j$ is connected to $e_j'$ via plaquettes entirely contained in $S_j.$
\end{proof}

To show the existence of a region where the Marcu--Fredenhagen ratio limits to zero, we compare the coefficients of exponential decay for the random cluster model with the coefficient of area law decay for the plaquette random cluster model.

\begin{prop}\label{proposition: region a2}
Assume the same hypotheses as in Theorem~\ref{thm:nontrivial}.  Let $p''=p''\paren{q,1,d},$ where $p''\paren{q,1,d}$ is given in Theorem~\ref{thm:areaperimeter}. Then, if \( p_2 > p'' \) and \( p_1  \) is sufficiently small (see~Figure~\ref{figure: zero region}), we have 
	\begin{align*}
		\lim_{n \to \infty} R(p_2,p_1,n) = 0.
	\end{align*} 
\end{prop}

\begin{proof}
	Let \( p_2,p_1 \in (0,1). \) Then, by Theorem~\ref{thm:stochdom}, we have  
	\begin{align*}
		\hat R(p_2,p_1,n) = \frac{\rho_{p_2,p_1}(V_{\gamma_n'})\rho_{p_2,p_1}(V_{\gamma_n''})}{\rho_{p_2,p_1}(V_{\gamma_n})} 
		=
		 \frac{\rho_{p_2,p_1}( V_{\gamma_n'} )^2 }{\rho_{p_2,p_1}( \eta \in V_{\gamma_n}) }
		\leq
		 \frac{\rho_{1,p_1}( V_{\gamma_n'} )^2 }{\rho_{p_2,0}(  V_{\gamma_n}) }.
	\end{align*}
	By combining~Theorem~\ref{thm:W=V} and~\cite[Theorem 7]{duncan2025topological}, it follows that if \( p_2 > p_c(q) \), then there is \( c({p_2})>0 \) such that
	\begin{align*}
		\rho_{p_2,0}(  V_{\gamma_n}) \geq e^{-c({p_2})|\gamma_n|} = e^{-2c({p_2})|\gamma_n'|}.
	\end{align*} 

Observe that $\paren{\Z^d,P_1}\in V_{\gamma_n'}$ if and only if $x_n$ is connected to $y_n$ in $P_1,$ where $x_n=({0,Rn,0,\ldots,0})$ and ${y_n=({Tn,R_n,0,\ldots,0}).}$ By Proposition~\ref{prop:10} and the sharpness of the phase transition for the random cluster model~\cite{duminil2019sharp}, there exists a $c'(P_1)$ so that 
    $$
        \rho_{1,p_1}( V_{\gamma_n'}) 
        \leq 
        e^{-c'(P_1)n}.
    $$
    Then, by Proposition~\ref{prop:c3}, we have that for $p_1$ sufficiently small,
    $$
        c'(P_1)
        \geq
        4\paren{1+\varepsilon}c({p_2}).
    $$

    From this, it follows that
    $$
        \rho_{1,p_1}( V_{\gamma_n'})
        \leq  
        e^{-4\paren{1+\varepsilon}c({p_2})n}=e^{-\paren{1+\varepsilon}c({p_2})|\gamma_n'|},
    $$


    where we used that
    $$
            |\gamma_n'| = 4n.
    $$

    The desired conclusion immediately follows.
    %
%
%
\end{proof}

\begin{figure}[h] 
\begin{subfigure}[t]{0.3\textwidth}

\begin{tikzpicture}[scale=.65]
    \begin{axis}[
        axis x line=middle,axis y line=middle, ymin=-.05, ymax=1.1, xmin=-.05, xmax=1.1, xtick={1}, ytick={1}, 
        xlabel={$p_2$}, ylabel={$p_1$},
        x label style={at={(axis description cs:1.5,0.08)},scale=.8},
        y label style={at={(axis description cs:0.08,1.5)},scale=.8}, 
        every axis/.append style={font=\tiny}, ]

\addplot[domain=0.3:1, samples=201, name path=E, draw=none]{0};
\addplot[domain=0.3:1, samples=201, name path=F, draw=none]{.5*(x-.3)^(1/2)};

\addplot [fill=blue, fill opacity = .15] fill between [of=E and F];


\end{axis}


	\end{tikzpicture}
	\caption{The region in Proposition~\ref{proposition: region a2}.}\label{figure: zero region}
    \end{subfigure}
    \hspace{1em}
    \begin{subfigure}[t]{0.3\textwidth} 

\begin{tikzpicture}[scale=.65]

\begin{axis}[
        axis x line=middle,axis y line=middle, ymin=-.05, ymax=1.1, xmin=-.05, xmax=1.1, xtick={1}, ytick={1}, xlabel={ $p_2$}, ylabel={ $p_1$},
        x label style={at={(axis description cs:1.5,0.08)},scale=.8},
        y label style={at={(axis description cs:0.08,1.5)},scale=.8},
    every axis/.append style={font=\tiny}, ]
 
\addplot[domain=0:1, samples=201, name path=A,draw=none]{1};  

\addplot[domain=0:1, samples=201, name path=B,draw=none]{0.8};

\addplot [fill=blue, fill opacity = .15] fill between [of=A and B];

\end{axis} 

	\end{tikzpicture}
	\caption{The region in Proposition~\ref{prop: nonzero 3}.} \label{figure: nonzero region 3}
    \end{subfigure}
    \hspace{1em}
    \begin{subfigure}[t]{0.3\textwidth} 

\begin{tikzpicture}[scale=.65]

\begin{axis}[
        axis x line=middle,axis y line=middle, ymin=-.05, ymax=1.1, xmin=-.05, xmax=1.1, xtick={1}, ytick={1}, xlabel={ $p_2$}, ylabel={ $p_1$},
        x label style={at={(axis description cs:1.5,0.08)},scale=.8},
        y label style={at={(axis description cs:0.08,1.5)},scale=.8},
    every axis/.append style={font=\tiny}, ]
        
\addplot[domain=0.08:0.081, samples=201, name path=E, draw=none]{1000*(x-0.08)};
\addplot[domain=0.0:0.001, samples=201, name path=F, draw=none]{1000*x};

\addplot [fill=blue, fill opacity = .15] fill between [of=E and F];

\end{axis}

	\end{tikzpicture}
	\caption{The region in Proposition~\ref{prop: nonzero 2}.} \label{figure: nonzero region 2} 
    \end{subfigure}
\end{figure}

To provide proofs for regions where the  Marcu--Fredenhagen Ratio is asymptotically bounded away from zero, we require a technical result that allows us to compare the value of $\rho^2\paren{P,Q}$ after ``switching'' maximal strongly connected components of $P_2$ and $Q_2.$ This will follow quickly from the next lemma. While we require this for $i=1$ only, we include a more general statement in case it will be useful in other contexts. 
 
\begin{lemma}\label{lemma:switch}
    Let $Z$ be a simply connected, finite cell complex. For a subcomplex $C$ of $Z,$ denote by $C^{\brac{j}}$ the $j$-dimensional percolation subcomplex containing all $j$-cells incident to an $(j+1)$-cell of $C.$
    If $\paren{X,A}$ and $\paren{Y,B}$ is a pair of subcomplexes of $Z$ of dimensions $(i+1,i)$ satisfying that 
    $$
        X\cap Y 
        = 
        Z^{(i)} \quad \text{and} \quad (X^{\brac{i}}\cap B) \cup (Y^{\brac{i}}\cap A)\cup \paren{A\cap B} = Z^{\paren{i-1}},
    $$
    then
    $$ 
        \dim H^i\paren{X\cup Y,A\cup B}
        =
        \dim H^i\paren{X,A}+\dim H^i\paren{Y,B}-\dim C^i\paren{Z} .
    $$
\end{lemma}

\begin{proof}
    Consider the following part of the Mayer--Vietoris sequence for cohomology
    \begin{equation}\label{eq: diagram}
        C^i(Z)\xleftarrow{\xi} H^i(X,A)\oplus H^i(Y,B)\leftarrow H^i(X\cup Y, A\cup B)\leftarrow 0\leftarrow \dots
    \end{equation}
     (see Section~\ref{sec:MV}), where we substituted
    $$
        H^{i-1}({X\cap Y,A\cap B})
        =
        H^{i-1}({Z^{(i)},Z^{\paren{i-1}}})=0
    $$
    and
    $$
        H^i(X\cap Y,A\cap B)
        =
        H^i({Z^{(i)},Z^{\paren{i-1}}})
        =
        C^i\paren{Z}
    $$
    and $\xi=\phi_X-\phi_Y.$
    The claim will follow if we show that the map $\xi$ in~\eqref{eq: diagram} is surjective. To this end, recall that 
    $$
        \xi\paren{\brac{f_1}\oplus \brac{f_2}} = f_1\mid_{X\cap Y}-f_2\mid_{X\cap Y}.
    $$
 
Let $\hat{X}^i=B\cup Y^{\brac{i}}$ and let $\tilde{Y}^i$ be the $i$-dimensional percolation subcomplex of $Z$ containing the remaining $i$-cells. For $f\in C^i\paren{Z}$ let $g_X, g_Y\in C^i\paren{Z}$ be obtained from  $f\mid_{\hat{X}^i}$ and $f\mid_{\tilde{Y}^i}$ by extending by zero, so $f=g_X+g_Y.$ Since $\hat{X}^i$ contains no $i$-cell of $A,$ it follows that $g_X\in C^i\paren{X,A}.$ Also, by construction, $g_X$ vanishes on all $i$-cells incident to an $(i+1)$-cell of $X$ so  $g_X\in Z^i\paren{X,A}.$ For similar reasons, $g_Y\in Z^i\paren{Y,B}.$  Then
$$\xi\paren{g_X,-g_Y}=g_X\mid_{X\cap Y}+g_Y\mid_{X\cap Y}=\alpha$$
so $\xi$ is surjective.
\end{proof}

\begin{corollary}
\label{cor:switching}
Let $\paren{X_1,A_1},$ $\paren{Y_1,B_1},$ $\paren{X_2,A_2},$ and $\paren{Y_2,B_2}$ be four pairs of percolation subcomplexes which are such that that the hypotheses of the previous lemma are satisfied when taking
$$\paren{\paren{X,A},\paren{Y,B}}=\paren{\paren{X_i,A_i},\paren{Y_j,B_j}}$$
for any $\paren{i,j}\in\set{1,2}\times \set{1,2}.$ Then
\begin{align*}
    &\rho\paren{X_1\cup Y_1,A_1\cup B_1}\rho\paren{X_2\cup Y_2,A_2\cup B_2}\\&\qquad=\rho\paren{X_1\cup Y_2,A_1\cup B_2}\rho\paren{X_2\cup Y_1,A_2\cup B_1}.
\end{align*}
\end{corollary}

\begin{proof}
    The desired conclusion immediately follows by noting that
    \begin{align*}
        &\dim H_i\paren{X_1\cup Y_1,A_1\cup B_1;\,\Z_q}+ \dim H_i\paren{X_2\cup Y_2,A_2\cup B_2;\,\Z_q}
        \\&\qquad=\dim H_i\paren{X_1\cup Y_2,A_1\cup B_2;\,\Z_q}+\dim H_i\paren{X_2\cup Y_1, A_2\cup B_1;\,\Z_q}.
    \end{align*}
\end{proof}

\begin{prop}\label{prop: nonzero 3}Assume the same hypotheses as in Theorem~\ref{thm:nontrivial}. If \( p_1  \) is sufficiently large  (see~Figure~\ref{figure: nonzero region 3}), then \[ \liminf_{n \to \infty} R(p_2,p_1,n) >0. \]
\end{prop}


\begin{proof}
Say that two $i$-cells $\sigma_1$ and $\sigma_2$ of a cell complex $Y$ are \emph{strongly connected} if there is a path $\tau_0=\sigma_1,\tau_2,\ldots,\tau_k=\sigma_2$ of $i$-cells of $Y$ between them so that $\tau_{j}$ and $\tau_{j+1}$ intersect in an $(i-1)$-cell for $j=0,\ldots,k-1.$ Call the resulting graph $\tilde{G}\brac{Y}.$ Given a pair of percolation sub-complexes \( P=(P_2,P_1) \) where \( P_2\) has dimension \(2\) and \( P_1 \) has dimension \( 1 \) respectively, let \( G=G[P] \) be the induced subgraph of $\tilde{G}\brac{P_2}$ on the collection of $2$-cells $\sigma$ of $P_2$ so that $\partial\sigma \nsubseteq P_1.$  

For path \( \gamma \), let \( G_\gamma[P]\) be the restriction of \( G[P] \) to its connected components that contain at least one $2$-cell incident to an edge of $\gamma.$  Denote the vertex set of $G_\gamma[P]$ by   $V( G_\gamma[P])$ and let $V_{1}( G_\gamma[P])$ be the collection of edges of $P_1$ that are either contained in $\gamma$ or are incident to at least one $2$-cell in $G_\gamma[P].$ Finally, to simplify notation, we introduce the following conventions. If $P=\paren{P_2,P_1}$ and $Q=\paren{Q_2,Q_1}$ are percolation subcomplexes of a cell complex $X$, denote by
$P\cup Q$ the pair $\paren{P_2\cup Q_2,P_1\cup Q_1},$ $P\cap Q$ the pair $\paren{P_2\cap Q_2,P_1\cap Q_1},$ and $P\setminus Q$ the pair $\paren{P_2\setminus Q_2,P_1\setminus Q_1}.$  

Let \( P^\gamma = (P_2^\gamma ,P_1^\gamma ),\) where $P_2^\gamma$ is the is the $2$-dimensional percolation subcomplex containing the  $2$-cells $V( G_\gamma[P])$ and $P_1^{\gamma}$ is the $1$-dimensional percolation subcomplex containing the $1$-cells $V_{1}( G_\gamma[P]).$
    For two disjoint paths \( \gamma' \) and \( \gamma'' \) and two pairs of percolation subcomplexes $P=\paren{P_2,P_1}$ and $Q=\paren{Q_2,Q_1}$ of the appropriate dimensions set 
    \[
        \mathcal{E} = \mathcal{E}_{\gamma',\gamma''} = \bigl\{ P,Q \colon  V(G_{\gamma'}[P \cup Q]) \cap V(G_{\gamma''}[P \cup Q]) = \emptyset \bigr\}.
    \]
    Note that if  $A=\Xi^{i+1}\times \Xi^{i}$ denotes the trivial event then 
    \begin{align*}
         \bigl((V_{\gamma'} \cap V_{\gamma''})\times A\bigr) \cap \mathcal{E} 
         =
         \bigl(V_{\gamma'+ \gamma''} \times A \bigr)\cap \mathcal{E}.
    \end{align*}
    We now show that
	\begin{equation}\label{eq: coupling 2}
		\hat R(p_2,p_1,n) \geq \rho^2   \bigl((P,Q) \in \mathcal{E} \mid P \in V_{\gamma_n} \bigr).
    \end{equation}
    
    To this end, for \( (P,Q) \in \mathcal{E},\) let
    \begin{equation*}
        \begin{cases}
            \varphi_1(P,Q) =  \bigl(Q  \cap (P\cup Q)^{\gamma_n''}\bigr)
            \cup
            \bigl(P \smallsetminus (P\cup Q)^{\gamma_n''} \bigr)  
            \cr  
            \varphi_2(P,Q) =  
            \bigl(P  \cap (P\cup Q)^{\gamma_n''} \bigr)
            \cup
            \bigl(Q \smallsetminus (P\cup Q)^{\gamma_n''}  
            \bigr).
        \end{cases}
    \end{equation*}
    Let $\rho_N=\rho_{p_2,p_1,q,1,\Lambda_N^d}.$  By applying Corollary~\ref{cor:switching} with $\paren{X_1,A_1}=P\smallsetminus \paren{P\cup Q}^{\gamma_n''},$  $\paren{Y_1,B_1}=P\cap \paren{P\cup Q}^{\gamma_n''},$ $\paren{X_2,A_2}=Q\smallsetminus \paren{P\cup Q}^{\gamma_n''},$ and  $\paren{Y_2,B_2}=Q\cap \paren{P\cup Q}^{\gamma_n''},$ we find that  if \( (P,Q) \sim \rho_N^2|_{\mathcal{E}} \), then 
	\begin{align*}
		&\bigl(
        \varphi_1(P,Q), \varphi_2(P,Q)
        \bigr) \sim \rho_N^2 |_{\mathcal{E}}.
	\end{align*} 
    By taking $N\to \infty,$ we obtain the same statement for the infinite volume measure $\rho.$ 
	Moreover, since for \( (P,Q) \in \mathcal{E}\) we have
	\begin{align*}
		&(P,Q) \in ( V_{\gamma_n'} \times V_{\gamma_n''} ) 
		\Leftrightarrow
		\bigl( \varphi_1(P,Q),\varphi_2(P,Q) \bigr) \in \bigl(( V_{\gamma_n} \cap V_{\gamma_n'})\times A \bigr),
	\end{align*}
	it follows that
	\begin{align*}
		&R(p_2,p_1,n) = \frac{\rho(  V_{\gamma_n'}) \rho( V_{\gamma_n''})}{\rho (V_{\gamma_n})}
        =
        \frac{\rho^2(P \in V_{\gamma_n'},\, Q \in V_{\gamma_n''})}{\rho^2(P \in V_{\gamma_n})}
        \\&\qquad\geq 
        \frac{\rho^2(P \in V_{\gamma_n'},\, Q \in V_{\gamma_n''},\, (P,Q) \in \mathcal{E})}{\rho^2(P \in V_{\gamma_n})}
        =
        \frac{\rho^2(P \in V_{\gamma_n'} \cap V_{\gamma_n''},\,\, (P,Q) \in  \mathcal{E})}{\rho^2(P \in V_{\gamma_n})}
        \\&\qquad=
        \frac{\rho^2(P \in V_{\gamma_n' + \gamma_n''},\,\, (P,Q) \in  \mathcal{E})}{\rho^2(P \in V_{\gamma_n})}
		=
		\rho^2 ( (P,Q) \in \mathcal{E}  \mid  P \in V_{\gamma_n} ,\, Q \in  A).
	\end{align*}
	This concludes the proof of~\eqref{eq: coupling 2}.

    We now give a lower bound of the right-hand side of~\eqref{eq: coupling 2} by showing that if \( p_1 \) is sufficiently large, then 
	\begin{equation}\label{eq: part 2 2}
        \liminf_{n \to \infty} \rho^2 ( (P,Q) \in \mathcal{E} \mid P \in V_{\gamma_n}) > 0 .
    \end{equation}
    To this end, given a pair \( P=(P_2,P_1) \) of percolation sub-complexes, let  \( \mathcal{E}'\) be the event that the set \( \tilde P \coloneqq \{ \sigma \in X^{\brac{2}} \colon \partial \sigma \subseteq P_1 \} \) separates \( \gamma_n'\) and \( \gamma_n''\), in the sense that any connected set of $2$-cells adjacent to both \( \gamma_n'\) and \( \gamma_n''\) must intersect \( \tilde P\). Then \( \mathcal{E} \supseteq \mathcal{E}',\) and hence
\begin{align*}
    &\rho^2 ( (P,Q) \in \mathcal{E}_{\gamma_n',\gamma''_n} \mid P \in V_{\gamma_n})
    \\&\qquad \geq 
    \rho^2 ( P\cap Q \in \mathcal{E}' \mid P \in V_{\gamma_n+\gamma_n'}).
\end{align*}
Since \( P \in V_{\gamma_n''+\gamma_n''}\) and \( P\cap Q \in \mathcal{E}'\) are both  increasing events and \( \rho^2 \) is positively associated, we can bound the right-hand side of the previous equation from below by
\begin{align*}
    &
    \rho^2 ( P\cap Q \in \mathcal{E}' \mid P \in V_{\gamma_n+\gamma_n'})
    \geq 
    \rho_{0,p_1,1}^2 ( P\cap Q \in \mathcal{E}' )
    = 
    \rho_{0,p_1^2,1} ( P \in \mathcal{E}' ).
\end{align*}
    where the second inequality follows from Theorem~\ref{thm:stochdom}. From this, we obtain~\eqref{eq: part 2 2}.
	Combining~\eqref{eq: coupling 2} and~\eqref{eq: part 2 2}, the desired conclusion immediately follows.
\end{proof}

We note that the point where the preceding proof fails for $i\geq 2$ is in establishing that  $\rho_{0,p_1^2,1} ( P \in \mathcal{E}' )$ is bounded away from zero.

Before establishing the behaviour of the Marcu-Fredenhagen ratio in the region in~\ref{figure: nonzero region 2}, we state and prove a discrete isoperimetric inequality. 
\begin{lemma}\label{lemma:isoperimetric}
For any $1$-cycle $\gamma\in Z_1\paren{\Z^d;\Z_q}$
$$\mathrm{Area}(\gamma)\leq  \frac{d-1}{8d} |\gamma|^2$$
where $|\gamma|$ denotes the number of edges in the support of $\gamma.$
\end{lemma}
\begin{proof}
    We will prove the statement using induction on the dimension $d.$ We abuse notation and use the symbol $\gamma$ refer to both the loop itself and the corresponding element of $Z^1\paren{\Z^d;\,\Z_q}.$  Suppose that $|\gamma|=T$ and that $S$ of the $T$ edges in its support are in the direction of the $x_d$ axis. By translating $\gamma$ if necessary, we may assume that it intersects the hyperplane $x_d=0$ and is contained in the slab  $\Z^{d-1}\times \brac{-\lfloor{\frac{S}{4}} \rfloor, \lfloor{\frac{S}{4}} \rfloor}.$ Let $Y$ be the projection of the support of $\gamma$ onto the hyperplane $x_d=0,$  let $C$ be the cylinder $Y\times \brac{-\lfloor{\frac{V}{4}} \rfloor, \lfloor{\frac{V}{4}} \rfloor},$ and let $C'$ be the union of the bounded components of the complement  $C\setminus \paren{\gamma \cup Y}.$ Since the inclusion $Y\hookrightarrow C$ induces an isomorphism on homology, there is $\tau\in C_2\paren{C;\Z_q}$ and a $\gamma'\in Z_1\paren{Y;\,\Z_q}$ so that $\gamma=\gamma'+\partial \tau.$ In fact, we may find a such a $\tau$ that is supported on $C'.$

    We now bound the area of $C'$ from above and below. Towards that end, we write $C'$ as the union of contributions from the edges of $\gamma.$ Let $e$ be an edge of $\gamma$ that is parallel to the hyperplane $x_d=0.$ We say that a cell $\sigma$ of~$C$ is in between $\gamma$ and $Y$ if the projections of $e$ and $\sigma$ onto the hyperplane $x_d=0$ coincide and the~$x_d$-coordinates of the points of $\sigma$ are between those of $\gamma$ and $e.$ If there is no other edge $e'$ of $\gamma$ so that $e$ is between $e'$ and $Y,$ we set $C_{e}$ to consist of all $2$-cells between $e$ and $\gamma.$ Otherwise, we set $C_{e}=\varnothing.$ Then $C'\subset \bigcup_{e\in \gamma}C_e$, and hence
    $$
        \mathrm{Area}\paren{\tau} \leq  \mathrm{Area}\paren{C'}\leq \sum_{e\in \gamma} \abs{C_e}\leq \frac{T\paren{T-S}}{4}.
    $$
    From this, it follows that 
    $$
        \mathrm{Area}(\gamma) \leq \frac{T\paren{T-S}}{4}+\mathrm{Area}_{d-1}\paren{\gamma'},
    $$
    where $\mathrm{Area}_{d-1}\paren{\gamma'}$ is the area of $\gamma'$ as an element of $Z_1\paren{\Z^{d-1};\,\Z_q}.$ 
    Therefore, if we set
    $$
        f_d(T) = \max_{\gamma:|\gamma|\leq T}\mathrm{Area}(\gamma),
    $$
    then, by induction, we have
$$f_d\paren{T}\leq  \frac{T\paren{T-S}}{4}+f_{d-1}\paren{T-S}\leq \frac{T\paren{T-S}}{4}+\frac{d-2}{8\paren{d-1}}\paren{T-S}^2.$$
The function on the right-hand side of the previous equation is maximized when $S=T/d$; plugging that in yields the desired formula.
\end{proof}

\begin{prop}\label{prop: nonzero 2}Assume the same hypotheses as in Theorem~\ref{thm:nontrivial}. 
	If $i=1$ and \( p_2  \) is sufficiently small (see Figure~\ref{figure: nonzero region 2}), then \[ \liminf_{n \to \infty} R(p_2,p_1,n) >0. \]
\end{prop}


\begin{proof}

We reuse notation from the previous proof except that we redefine $P^{\gamma}$ below. We now argue that the right-hand side of~\eqref{eq: coupling 2}, i.e., 
$$\rho^2   \bigl((P,Q) \in \mathcal{E} \mid P \in V_{\gamma_n} \bigr)$$
can be bounded from below when $p_2$ is sufficiently small, uniformly in \( p_1 \) and \( |\gamma_n|,\) where we recall that $\gamma_n=\gamma_n'+\gamma_n''.$ To this end, we show that when $P$ is sampled from the CPP conditional on the event $V_{\gamma_n}$ and $p_2$ is small, the strongly connected components of $P_2$ exhibit exponential decay uniformly in \( p_1 \) and \( |\gamma_n|\). The desired conclusion then follows from standard arguments. Since $\rho$ is stochastically decreasing in $p_1,$ we may assume that $p_1<1/2.$

For an edge $\epsilon$ and a pair of percolation subcomplexes $P=\paren{P_2,P_1}$, let $P_2^{\epsilon}$ be the $2$-dimensional percolation subcomplex whose $2$-cells are the strongly connected component of $\epsilon$ in $P_2$ and set
$$P^{\epsilon}=(P_2^{\epsilon},( P_2^{\epsilon))^{\brac{1}}\cap P_1}.$$ Also, let  $\Gamma_n=\paren{X^{\paren{1}},\support\gamma_n \cup X^{1}}.$
Fix $\epsilon\in\support\gamma_n,$ let $R=\paren{R_2,R_1}$ satisfy $R=R^{\epsilon},$ and let
    $$
        Q = 
        Q(R)
        =
        R\cup \bigl(\Gamma_n\setminus ( P_2^{\epsilon})^{\brac{1}} \bigr).$$
    Also, set
    $$
        P'
        =
        ( P\smallsetminus P^{\epsilon}) \cup \bigl( (P_2^{\epsilon})^{\brac{1}} \cap \Gamma_n\bigr).
    $$
By construction, if $P\in V_{\gamma_n}$, then $ Q\paren{P^{\epsilon}},P'\in V_{\gamma_n}.$ 
Moreover, by applying Corollary~\ref{cor:switching} with  ${\paren{X_1,A_1}={P\smallsetminus P^{\epsilon}},}$  ${\paren{Y_1,B_1}=\paren{P_2^{\epsilon}}^{\brac{1}} \cap \Gamma_n,}$ ${\paren{X_2,A_2}=\Gamma_n\setminus \paren{P_2^{\epsilon}}^{\brac{1}},}$ and  $\paren{Y_2,B_2}=P^{\epsilon},$ we obtain
    $$
        \rho_N(P') \rho_N(Q)
        =
        \rho_N(p)\rho_N(\Gamma_n).
    $$
    This implies in particular that
    \begin{align*}
        \rho_N( P^\epsilon   = R \mid P \in V_{\gamma_n}) &=
        \frac{\sum_{P \in V_{\gamma_n}} \rho_N(P) \mathbf{1}(P^\epsilon = R)}{\sum_{P \in V_{\gamma_n}} \rho_N(P)} 
        \leq 
        \frac{\sum_{P \in V_{\gamma_n}} \rho_N(P) \mathbf{1}(P^\epsilon = R)}{\sum_{P \in V_{\gamma_n}} \rho_N(P')\mathbf{1}(P^\epsilon = R)} 
        \\
       & = 
       \rho_N(Q)/\rho_N
       \leq \paren{\frac{p_2}{1-p_2}}^{\abs{R_2}}\paren{\frac{p_1}{1-p_1}}^{\abs{R_1}-\abs{\gamma_n}}\frac{|{H^1({\Lambda_N^{\paren{1}};\,\Z_q})}|}{|H^1({\Lambda_N^{\paren{1}};\,\Z_q})|-1}\\
       \intertext{and}\\
         \rho( P^\epsilon   = R \mid P \in V_{\gamma_n})&\leq \paren{\frac{p_2}{1-p_2}}^{\abs{R_2}}\paren{\frac{p_1}{1-p_1}}^{\abs{R_1}-\abs{\gamma_n}}.
        \end{align*}

We now consider two cases. If
$\abs{R_1}\geq \abs{\gamma_n}$
then
$$\rho( P^\epsilon   = R \mid P \in V_{\gamma_n})\leq \paren{\frac{p_2}{1-p_2}}^{\abs{R_2}}$$
where we used that $p_1<1/2.$ If instead $\abs{R_1}< \abs{\gamma_n}$, we argue that any witness of $V_{\gamma_n}$ must have many two-cells in $R_2.$ Since $\paren{R_2,R_1}\in V_{\gamma_n}$ there there is a $1$-cycle $\tau\in Z_1\paren{R_1;\,\Z_q}$ and a $2$-chain $\sigma\in C_2\paren{R_2;\Z_q}$ so that $\partial \sigma=\gamma_n-\tau.$ Recall that the area of a $1$-cycle is the number of $2$-cells in the support of a minimal null-homology. Applying Lemma~\ref{lemma:isoperimetric} to $\tau$ yields that
$$\abs{R_2}\geq \mathrm{Area}(\Gamma_n)-\mathrm{Area}\paren{\tau}\geq \hat{c}\paren{n^2 -\abs{R_1}^2},$$
where $\hat{c}=\frac{d-1}{8d}.$ 
Then
$$\frac{\abs{\gamma_n}-\abs{R_1}}{\abs{R_2}}\leq \frac{4n-\abs{R_1}}{n^2 -\hat{c}\abs{R_1}^2}=\frac{1}{\nicefrac{n}{4}+\hat{c}\abs{R_1}}$$
so
    \begin{align*}
        \rho( P^\epsilon   = R \mid P \in V_{\gamma_n})  \leq & \paren{\frac{p_2}{1-p_2}}^{\abs{R_2}}\paren{\frac{p_1}{1-p_1}}^{\abs{R_1}-\abs{\gamma_n}}\paren{\paren{\frac{p_2}{1-p_2}} \paren{\frac{1-p_1}{p_1}}^{\frac{1}{\nicefrac{n}{4}+\hat{c}\abs{R_1}}}}^{\abs{R_2}}
        \\
        \leq & \paren{\frac{p_2+o_n(1)}{1-p_2+o_n(1)}}^{\abs{R_2}}.
    \end{align*}

Let $\Xi^2_{\epsilon,N}$ denote the collection of strongly connected subsets of $\Xi^2$ that are incident to $\epsilon$ and contain $N$ $2$-cells. By standard arguments, there is a $\lambda>0$ so that
$|\Xi^2_{\epsilon,N}|<e^{\lambda N}.$ 
Thus
\begin{align*}
\rho\paren{\abs{P_2^{\epsilon}}=N\mid P \in V_{\gamma_n}}&=\sum_{R_2\in \Xi^2_{\epsilon,N}}\sum_{R_1\subset R_2^{\brac{1}}}\rho(p)\mathbf{1}\paren{P^\epsilon = \paren{R_2,R_1}, P\in V_{\gamma}}\\
&<  e^{\lambda N} 2^{4 N} \paren{\frac{p_2+o_n(1)}{1-p_2+o_n(1)}}^{\abs{R_2}}
\end{align*}
which decays exponentially in $N$ when $p_2$ is sufficiently small and $n$ is sufficiently large. 
\end{proof}

\section*{Acknowledgments}
We'd like to thank Paul Duncan for interesting discussions and comments on an earlier draft of this paper, and Fedor Manin for suggesting the proof of Lemma~\ref{lemma:isoperimetric}.

\newpage
\bibliographystyle{alpha}
\bibliography{refs}

@article{F24,
    author = "Forsström, Malin Palö",
    title = "{{W}ilson Lines in the Abelian Lattice {H}iggs Model}",
    doi = "10.1007/s00220-024-05128-x",
    journal = "Communications in Mathematical Physics",
    volume = "405",
    issue = "11",
    number = "275",
    year = "2024"
}

@article{A25,
  title={Geometric Analysis of {I}sing Models, Part {III}},
  author={Aizenman, Michael},
  journal={Mathematical Physics, Analysis and Geometry},
  volume={28},
  number={4},
  pages={32},
  year={2025},
  publisher={Springer}
}

@article{G16,
    author = "Grimmett, Geoffrey R.",
    title = "{Correlation inequalities for the {P}otts model}",
    doi = "10.2140/memocs.2016.4.327",
    journal = "Mathematics and Mechanics of Complex Systems",
    volume = "4",
    pages = "327-334",
    year = "2016"
}

@book{S95,
  address = {New York},
  author = {Spanier, Edwin Henry},
  publisher = {Springer},
  title = {Algebraic topology},
  year = 1966
}

@book{H02,

  address = {Cambridge},
  author = {Hatcher, Allen},

  publisher = {Cambridge University Press},
  title = {Algebraic topology},

  year = 2002
}

@article{FKG71,
  title={Correlation inequalities on some partially ordered sets},
  author={Fortuin, Cees M. and Kasteleyn, Pieter W. and Ginibre, Jean},
  journal={Communications in Mathematical Physics},
  volume={22},
  pages={89--103},
  year={1971},
  publisher={Springer}
}

@book{G06,
  address = {New York},
  author = {Grimmet, Geoffrey},
  publisher = {Springer},
  title = {The Random-Cluster Model},
  year = 2006
}

@article{FK72,
title = {On the random-cluster model: {I}. Introduction and relation to other models},
journal = {Physica},
volume = {57},
number = {4},
pages = {536-564},
year = {1972},
issn = {0031-8914},
doi = {https://doi.org/10.1016/0031-8914(72)90045-6},
url = {https://www.sciencedirect.com/science/article/pii/0031891472900456},
author = {Fortuin, Cees M. and Kasteleyn, Pieter W.}
}

@article{NSS21,
  title = {Self-Dual Criticality in Three-Dimensional {$\mathbb{Z}_{2}$} Gauge Theory with Matter},
  author = {Somoza, Andr\'es M. and Serna, Pablo and Nahum, Adam},
  journal = {Physical Review X},
  volume = {11},
  issue = {4},

  year = {2021},
  publisher = {American Physical Society},

}

@article{HS91,
  title = {Are sponge phases of membranes experimental gauge-{H}iggs systems?},
  author = {Huse, David A. and Leibler, Stanislas},
  journal = {Physical Review Letters},
  volume = {66},
  issue = {4},
  pages = {437--440},

  year = {1991},

  publisher = {American Physical Society},

}

@misc{F25a,
    author = {Forsstr\"om, Malin Palö and Viklund, Fredrik},
    title = "{A current expansion for {I}sing lattice gauge theory}",
    eprint = "2502.19942",
    archivePrefix = "arXiv",
    year = "2025"
}

@article{B88,
title = {Determination of the phase structure of the four-dimensional coupled gauge-{H}iggs {P}otts model},
journal = {Physics Letters B},
volume = {207},
number = {3},
pages = {300-304},
year = {1988},

author = {Baig, Mari\`a},

}

@article{duncan2025topological,
  title={Topological Phases in the Plaquette Random-Cluster Model and Potts Lattice Gauge Theory},
  author={Duncan, Paul and Schweinhart, Benjamin},
  journal={Communications in Mathematical Physics},
  volume={406},
  number={6},
  pages={145},
  year={2025},
  publisher={Springer}
}

@article{duncan2025sharp,
  title={A Sharp Deconfinement Transition for {P}otts Lattice Gauge Theory in Codimension Two},
  author={Duncan, Paul and Schweinhart, Benjamin},
  journal={Communications in Mathematical Physics},
  volume={406},
  number={7},
  pages={164},
  year={2025},
  publisher={Springer}
}

@article{duncan2025homological,
	title={Homological percolation on a torus: plaquettes and permutohedra},
	url={arxiv.org/abs/10.48550/arXiv.2011.11903},
	author={Duncan, Paul and Kahle, Matthew and Schweinhart, Benjamin},
	year={2025},
    journal={Annales de l'Institut Henri Poincar{\'e}, Probabilit{\'e}s et Statistiques},
}

@article{prcm-24,
  title={Some properties of the plaquette random-cluster model},
  author={Duncan, Paul and Schweinhart, Benjamin},
  journal={Mathematical Physics, Analysis and Geometry},
  volume={29},
  number={1},
  pages={2},
  year={2026},
  publisher={Springer}
}

@mastersthesis{shklarov,
	title={The {E}dwards--{S}okal Coupling for the {P}otts Higher Lattice Gauge Theory on $\mathbb{Z}^d$},
	author={Shklarov, Yakov},
	year={2023},
	school={The University of Victoria}
}

@article{bernoulli-cell-complexes,
	title={Tutte polynomials and random-cluster models in {B}ernoulli cell complexes},
	volume={B59},
	journal={RIMS Kokyuroku Bessatsu},
	publisher={Research Institute for Mathematical Sciences, Kyoto University},
	author={Hiraoka, Yasuaki and Shirai, Tomoyuki},
	year={2016},
	month=jul,
	pages={289–304}
}

@article{FS79,
  title = {Phase diagrams of lattice gauge theories with {H}iggs fields},
  author = {Fradkin, Eduardo and Shenker, Stephen H.},
  journal = {Phys. Rev. D},
  volume = {19},
  issue = {12},
  pages = {3682--3697},
  year = {1979},
  month = {Jun},
  publisher = {American Physical Society},
  doi = {10.1103/PhysRevD.19.3682},
}

@misc{F25b,
      title={Ornstein--{Z}ernike decay of {W}ilson line observables in the free phase of the \( \mathbb{Z}_2 \) lattice {H}iggs model}, 
      author={Malin Palö Forsström},
      year={2025},
      eprint={2504.10909},
      archivePrefix={arXiv},
      primaryClass={math.PR},
      url={https://arxiv.org/abs/2504.10909}, 
}

@article{hansen2025general,
  title={A General Coupling for {I}sing Models and Beyond},
  author={Hansen, Ulrik Thinggaard and Jiang, Jianping and Klausen, Frederik Ravn},
  journal={arXiv preprint arXiv:2506.10765},
  year={2025}
}

@article{zhang2020loop,
  title={Loop-cluster coupling and algorithm for classical statistical models},
  author={Zhang, Lei and Michel, Manon and El{\c{c}}i, Eren M and Deng, Youjin},
  journal={Physical Review Letters},
  volume={125},
  number={20},
  pages={200603},
  year={2020},
  publisher={APS}
}

@article {aizenman1983sharp,
    AUTHOR = {Aizenman, Michael and Chayes, Jennifer Tour and Chayes, Lincoln and Fr\"{o}hlich, J\"urg
              and Russo, Lucio},
     TITLE = {On a sharp transition from area law to perimeter law in a
              system of random surfaces},
   JOURNAL = {Communications in Mathematical Physics},
  FJOURNAL = {Communications in Mathematical Physics},
    VOLUME = {92},
      YEAR = {1983},
    NUMBER = {1},
     PAGES = {19--69},
      ISSN = {0010-3616},
   MRCLASS = {82A05 (60K35 81E25 82A42)},
  MRNUMBER = {728447},
MRREVIEWER = {J. Theodore Cox},
}

@book{kaczynski2004computational,
  title={Computational homology},
  author={Kaczynski, Tomasz and Mischaikow, Konstantin Michael and Mrozek, Marian},
  year={2004},
  publisher={Springer},
  address={Berlin Heidelberg}
}

@book{saveliev2016topology,
  title={Topology illustrated},
  author={Saveliev, Peter},
  year={2016},
  publisher={Independently published},
  address={}
}

@book{bredon2013topology,
  title={Topology and geometry},
  author={Bredon, Glen E},
  volume={139},
  year={2013},
  publisher={Springer Science \& Business Media}
}

@article{fredenhagen1988dual,
  title={Dual interpretation of order parameters for lattice gauge theories with matter fields},
  author={Fredenhagen, Klaus and Marcu, Mihail},
  journal={Nuclear Physics B-Proceedings Supplements},
  volume={4},
  pages={352--357},
  year={1988},
  publisher={Elsevier}
}

@article{bricmont1985statistical,
  title={Statistical mechanical methods in particle structure analysis of lattice field theories:({I}). General theory},
  author={Bricmont, Jean and Fr{\"o}hlich, J},
  journal={Nuclear Physics B},
  volume={251},
  pages={517--552},
  year={1985},
  publisher={Elsevier}
}

@article{p2024phase,
  title={The phase transition of the {M}arcu--{F}redenhagen ratio in the abelian lattice {H}iggs model},
  author={Forsstr{\"o}m, Malin Palö},
  journal={Electronic Journal of Probability},
  volume={29},
  pages={1--36},
  year={2024},
  publisher={The Institute of Mathematical Statistics and the Bernoulli Society}
}

@article{duminil2019sharp,
  title={Sharp phase transition for the random-cluster and {P}otts models via decision trees},
  author={Duminil-Copin, Hugo and Raoufi, Aran and Tassion, Vincent},
  journal={Annals of Mathematics},
  volume={189},
  number={1},
  pages={75--99},
  year={2019},
}

@article{stahl2026slow,
  title={Slow mixing and emergent one-form symmetries in three-dimensional {$\mathbb{Z}_2$} gauge theory},
  author={Stahl, Charles and Placke, Benedikt and Khemani, Vedika and Li, Yaodong},
  journal={arXiv preprint arXiv:2601.06010},
  year={2026}
}

@article{wegner,
  title={Duality in generalized {I}sing models and phase transitions without local order parameters},
  author={Wegner, Franz J},
  journal={Journal of Mathematical Physics},
  volume={12},
  number={10},
  pages={2259--2272},
  year={1971},
  publisher={American Institute of Physics}
}

@article{sw,
	title={Nonuniversal critical dynamics in {M}onte {C}arlo simulations},
	volume={58},
	DOI={10.1103/PhysRevLett.58.86},
	number={2},
	journal={Physical Review Letters},
	publisher={American Physical Society},
	author={Swendsen, Robert H. and Wang, Jian-Sheng},
	year={1987},
	month=jan
}

@article{edwards-sokal,
    title={Generalization of the {F}ortuin-{K}asteleyn-{S}wendsen-{W}ang representation and {M}onte {C}arlo algorithm},
    volume={38},
    DOI={10.1103/PhysRevD.38.2009},
    number={6},
    journal={Physical Review D},
    publisher={American Physical Society},
    author={Edwards, Robert G. and Sokal, Alan D.},
    year={1988},
    month=sep,
    pages={2009–2012}
}

@article{pizzimenti2025generalized,
  title={Generalized cluster algorithms for {P}otts lattice gauge theory},
  author={Pizzimenti, Anthony E and Duncan, Paul and Schweinhart, Benjamin},
  journal={arXiv preprint arXiv:2507.13503},
  year={2025}
}

@article{GLIOZZI2006120,
title = {The functional form of open Wilson lines in gauge theories coupled to matter},
journal = {Nuclear Physics B - Proceedings Supplements},
volume = {153},
number = {1},
pages = {120-127},
year = {2006},
note = {Proceedings of the Workshop on Computational Hadron Physics},
issn = {0920-5632},
doi = {https://doi.org/10.1016/j.nuclphysbps.2006.01.018},
url = {https://www.sciencedirect.com/science/article/pii/S0920563206000193},
author = {F. Gliozzi}
}

\end{document}